\documentclass[12pt]{amsart}
\usepackage[utf8]{inputenc}

\usepackage[margin=1.2in]{geometry}                % See geometry.pdf to learn the layout options. There are lots.
\usepackage{amsmath,amsthm,amssymb}
\usepackage{graphicx}
\usepackage{comment}
\usepackage{epstopdf}
\usepackage{tikz}
\usepackage{tikz-cd} 
\usepackage{xcolor}
\usepackage[colorlinks=true, pdfstartview=FitV, linkcolor=blue, citecolor=blue, urlcolor=blue]{hyperref}
\usepackage[nameinlink]{cleveref}
    \crefname{conj}{conjecture}{conjectures}
    \crefname{algocfline}{algorithm}{algorithms}
    
\usepackage{pgfplots, multicol, float, mathtools, extarrows}
\usepgfplotslibrary{fillbetween}
\usepackage[linesnumbered,ruled,vlined]{algorithm2e}
\SetArgSty{textnormal} % sets the font style of algorithm (especially in for-loop to non-italics

\makeatletter  % for 2020 subject classification
\@namedef{subjclassname@2020}{%
  \textup{2020} Mathematics Subject Classification}
\makeatother

\newcommand{\N}{{\mathbb N}}

\newcommand{\R}{{\mathbb R}}
\newcommand{\Q}{{\mathbb Q}}

\newcommand{\ba}{{\mathbf a}}
\newcommand{\bb}{{\mathbf b}}
\newcommand{\bc}{{\mathbf c}}
\newcommand{\be}{{\mathbf e}}
\newcommand{\bp}{{\mathbf p}}

\newcommand{\bt}{{\mathbf t}}
\newcommand{\bx}{{\mathbf x}}
\newcommand{\by}{{\mathbf y}}
\newcommand{\bz}{{\mathbf z}}
\newcommand{\bv}{{\mathbf v}}

\newcommand{\cL}{{\mathcal L}}

\def\conv{\operatorname{convex\ hull}}
\def\d{\operatorname{dist}}

\def\hype{\operatorname{hype}}

\def\np{\operatorname{np}}

\def\ov{\overline}

\newtheorem{thm}{Theorem}[section]

\newtheorem*{introthm*}{Theorem}

\newtheorem{cor}[thm]{Corollary}
\newtheorem{lem}[thm]{Lemma}
\newtheorem{prop}[thm]{Proposition}

\theoremstyle{definition}
\newtheorem{defn}[thm]{Definition}

\newtheorem{ex}[thm]{Example}

\theoremstyle{remark}
\newtheorem{rem}[thm]{Remark}

\numberwithin{equation}{section}  %% equation numbering

\definecolor{NordDarkBlack}{HTML}{2E3440}     % nord0
\definecolor{NordBlack}{HTML}{3B4252}         % nord1
\definecolor{NordMediumBlack}{HTML}{434C5e}   % nord2
\definecolor{NordBrightBlack}{HTML}{4C566A}   % nord3
\definecolor{NordWhite}{HTML}{E5E9F0}         % nord5
\definecolor{NordBrightWhite}{HTML}{ECEFF4}   % nord6
\definecolor{NordCyan}{HTML}{8FBCBB}          % nord7
\definecolor{NordBrightCyan}{HTML}{88C0D0}    % nord8
\definecolor{NordBlue}{HTML}{81A1C1}          % nord9
\definecolor{NordBrightBlue}{HTML}{5E81AC}    % nord10
\definecolor{NordRed}{HTML}{BF616A}           % nord11
\definecolor{NordOrange}{HTML}{D08770}        % nord12
\definecolor{NordYellow}{HTML}{EBCB8B}        % nord13
\definecolor{NordGreen}{HTML}{A3BE8C}         % nord14
\definecolor{NordMagenta}{HTML}{B48EAD}       % nord15

\pgfplotsset{posQuad/.append style={grid=both, xlabel={$x$}, ylabel={$y$}, line width=2pt, mark size=3pt,draw=NordWhite}}

\SetKwInput{KwInput}{Input}                % Set the Input for algorithms
\SetKwInput{KwOutput}{Output}

\title{Computing rational powers of monomial ideals}
\author[Dongre, Drabkin, Lim, Partida, Roy, Ruff, Seceleanu, Tang]{Pratik Dongre, Benjamin Drabkin, Josiah Lim, Ethan Partida, Ethan Roy, Dylan Ruff, Alexandra Seceleanu, Tingting Tang}
\thanks{}
\address{Indian Institute of Information Technology, Nagpur}
\email{mepratikdongre111@gmail.com}
\address{Singapore University of technology and Design}
\email{benjamin\string_drabkin@sutd.edu.sg}
\address{Johns Hopkins University}
\email{jlim76@jhu.edu}
\address{Brown University}
\email{ethan\string_partida@brown.edu}
\address{The University of Texas at Austin}
\email{ethanroy@utexas.edu}
\address{University of Toronto}
\email{dylan.ruff@mail.utoronto.ca}
\address{University of Nebraska--Lincoln}
\email{aseceleanu@unl.edu}
\address{San Diego State University}
\email{ttang2@sdsu.edu}


\begin{document}

\begin{abstract}
This paper concerns fractional powers of monomial ideals. Rational powers of a monomial ideal generalize the integral closure operation as well as recover the family of symbolic powers. They also highlight many interesting connections to the theory of convex polytopes. We provide multiple algorithms for computing the rational powers of a monomial ideal. We also introduce a mild generalization allowing real powers of monomial ideals. An important result is that given any monomial ideal $I$, the function taking a real number to the corresponding real power of $I$ is a step function which is left continuous and has rational discontinuity points.  \end{abstract}
 
 \maketitle
 
 %\tableofcontents

 \section{Introduction}
 
An ideal of the polynomial ring $R=K[x_1,\ldots, x_d]$ with coefficients in a field $K$ is a {\it monomial ideal} if it is generated by monomials.

In this paper, we study a notion of powers for monomial ideals, where the exponents are allowed to be real numbers as follows:  for $r\in \R$, $r>0$ we define the $r$-th {\em real power} of a monomial ideal, $I$, denoted  $\ov{I^r}$   to be the monomial ideal whose exponent set consists of (integer) lattice points in the $r$-th dilate of the Newton polyhedron of $I$; see \Cref{def:realpower}.  We emphasize that $\ov{I^r}$ is an ideal of the polynomial ring $R$, and in particular the monomial generators of $\ov{I^r}$ have natural number exponents. Thus our notion of real powers of ideals bears no overlap with work taking place in a ring where monomials are allowed to have real number exponents. Prominent examples of work in the latter context are \cite{IS,AS, Miller}.

%This notion arises as an algebraic counterpart for the operation of scaling the Newton polyhedron of an ideal by a positive real scalar. We term the resulting ideals {\em real powers} of monomial ideals. 

Our notion of real powers is inspired by, and in fact coincides when $r\in \Q$, with the notion of rational powers,  which can be defined for arbitrary ideals, and have appeared previously in the literature in \cite[\S10.5]{SH},  \cite{Knutson}, \cite{RushRational}, \cite{CiupercaRational}, \cite{CiupercaGolod}, \cite{LewisRational}. In these works, rational powers come up in contexts ranging from valuation theory to intersection theory  and have application to establishing the Golod property. In particular, \cite[Corollary 3.4]{LewisRational} establishes a strong connection between rational powers and the widely studied family of symbolic powers of monomial ideals. The above mentioned applications have motivated and inspired us to seek effective methods for handling rational powers from a computational standpoint. 

The focus of this paper is twofold. First, we handle the task of computing real powers of monomial ideals. One main result in this direction is \Cref{thm:boundeddistance}, where we show that the generators of a specified real power of a monomial ideal can be confined within a bounded convex region depending only on the exponent and the Newton polytope of the ideal. We complement this theoretical insight with a series of algorithms, \Cref{alg:Minkowski}, \Cref{alg:hyperrect}, \Cref{alg:improvedhyperrect}, and \Cref{alg:staircasealg} which exploit different features of the problem to provide practical solutions for computing real powers of monomial ideals.  

Our second aim is to study continuity properties of the exponentiation function where the base is a monomial ideal. Being able to do this provides motivation for working with real powers as opposed to the more common rational powers. We find that the exponentiation function is a step function with rational discontinuity points which we term jumping points. This leads to the conclusion that all distinct real powers of a fixed monomial ideal are given by rational exponents. Our main results on properties for the real exponentiation function of a monomial ideal are contained in \Cref{lem:I>r=I^t_n} (existence of right limits) \Cref{lem:leftcont} (left continuity), \Cref{cor:step} (step function), and \Cref{thm:jumpingrational} (jumping numbers).

Our paper is organized as follows. After introducing the notions of Newton polyhedron and integral closure in \cref{s:NP}, we turn our attention to real powers of monomial ideals in \cref{s:rp} and present algorithms capable of computing these ideals in \cref{s:alg}. We end with studying continuity properties and jumping numbers for exponentiation in \cref{s:jumping}.
 
\section{Background on integral closure and the Newton polyhedron}
\label{s:NP}

Let $\R$ and $\R_+$ denote the real numbers and non-negative real numbers respectively. We denote by $\N$ the set of non negative integers. 

 Let $R=K[x_1,\cdots,x_d]$ be a polynomial ring with coefficients in a field $K$.
Every monomial ideal $I$ in $R$ has a unique minimal monomial generating set denoted $G(I)$. This is a set of monomials that generates $I$ and such that no element of $G(I)$ divides another element of $G(I)$.
%By a minimal generating set we mean a set of generators for $I$ that is minimal either with respect to cardinality or to containment, since these two notions are equivalent by Nakayama's Lemma. 
It is customary to denote monomials in $R$ by the shorthand notation $\bx^\ba:=x_1^{a_1}\cdots x_d^{a_d}$, where $\ba\in\N^d$. The bijective correspondence between monomials $\bx^\ba$ and lattice points $\ba\in \N^n$ gives rise to convex geometric representations for monomial ideals, chief among which is the Newton polyhedron.

\begin{defn}
For any monomial ideal $I$ denote by $\cL(I)$ the set of exponent vectors of all monomials in $I$
\[\cL(I)=\{\ba \mid \bx^\ba\in I\}.\]
The {\em Newton polyhedron} of $I$, denoted $NP(I)$, is the convex hull of $\cL(I)$ in $\R^d$
\[
NP(I)=\conv \cL(I)= \conv ( \{\ba \mid \bx^\ba\in I\}).
\]
The {\em Newton polytope} of $I$, denoted $\np(I)$,  is the convex hull of the exponent vectors of a minimal monomial generating set  for $I$. 
\[
\np(I)=\conv ( \{\ba \mid \bx^\ba\in G(I)\}).
\]
\end{defn}

Notice that Newton polyhedra are unbounded, while Newton polytopes are bounded convex bodies. Both are lattice polyhedra, meaning that their vertices have integer coordinates. Their relationship can be described using the notion of Minkowski sum.

\begin{defn}
The {\em Minkowski sum} of subsets $A,B\subseteq \R^n$ is 
\[A+B=\{\ba+\bb \mid \ba \in A, \bb \in B\}.\]
We also write $A-B=\{\ba-\bb \mid \ba \in A, \bb \in B\}$.
\end{defn}

%\begin{rem}
%\label{rem:Minkowski_decomp}
The precise relationship between the Newton polyhedron and the Newton polytope of $I$, established for example in \cite[Lemma 5.2]{cooper2017symbolic}, is given by the Minkowski sum decomposition
\begin{equation}
    \label{eq:NP=np+orthant}
    NP(I)=\np(I)+\R^d_+,
\end{equation}
where  $\R^d_+=\{(a_1,\ldots,a_d)\in \R^d \mid a_i\geq 0\}$ denotes the positive orthant in $\R^d$. 

By the version of Carath\'eodory's theorem in \cite[Theorem 5.2]{cooper2017symbolic}, any point $\ba\in NP(I)$ is written as
\begin{equation}
    \label{eq:Caratheodory}
\ba=\lambda_1\bt_1+\cdots+\lambda_d\bt_d+c_1\be_1+\cdots+c_d\be_d,
\end{equation}
with $\lambda_i, c_j\geq 0, \sum_{i=1}^d \lambda_i=1$, $\bt_1, \ldots, \bt_d\in \np(I)$, and $\be_1,\ldots, \be_d$ standard basis vectors in $\R^d$. Thus one can reformulate equation \eqref{eq:NP=np+orthant} using coordinatewise inequalities as
\begin{equation}
    \label{eq:NP}
    NP(I)=\{\ba \in \R^d \mid \ba \geq \bb \text{ for some }\bb\in \np(I) \}
\end{equation}
%\end{rem}

While the containment $\cL(I)\subseteq NP(I)\cap \N^d$ holds by definition, in general the sets of lattice points $\cL(I)$ and $NP(I)\cap \N^d$ need not be equal.  We recall below that the set of lattice points in $NP(I)$ is in fact given by $NP(I)\cap \N^d=\cL(\overline{I})$, where $\overline{I}$ is the integral closure of $I$.

\begin{defn}
The {\em integral closure} of an ideal $I$ of a ring $R$ is the set of elements  $y\in R$ that satisfy an equation of integral dependence of the form
\[y^n +m_1y^{n-1} +\cdots +m_{n-1}y+m_n =0 \text{ where } m_i\in I^i, n\geq 1.\]
The integral closure of $I$ is denoted $\overline{I}$.
\end{defn}

\begin{rem}
\label{rem:intclosure}
It is shown in \cite{SH} that the description is significantly simpler if $I$ is a monomial ideal. In this case one can give an alternate definition for the integral closure 
\begin{equation}
\label{eq:intclalt}
\overline{I}=\left (\{ \bx^\ba \mid \bx^{n\ba}\in I^n \text{ for some } n\in \N\}\right).
\end{equation}
\end{rem}

We recall below how the integral closure of a monomial ideal $I$ can be described in terms of its Newton polyhedron. We also show that  the minimal generators of $\ov{I}$  lie at bounded lattice distance from the Newton polytope $\np(I)$.
In the following we use the notion of lattice (or taxicab) distance between points in $\ba,\bb \in \R^d$ defined as $\d(\ba,\bb)=\sum_{i=1}^d |a_i-b_i|$.

\begin{lem}
\label{lem:intclosure}
Let $I$ be a monomial ideal in $K[x_1,\ldots,x_d]$. Then
\begin{enumerate}
\item   $NP(I)\cap \N^d=\cL(\ov{I})$,
\item $NP(\ov{I})=NP(I)$,
 \item (compare \cite[Proposition 1.4.9]{SH}) if $\bx^\ba\in G(\overline{I})$, then there exists $\bb\in \np(I)$  such that  $\ba\geq \bb$ and
 \[\sum_{i=1}^d (a_i-b_i) \leq d-1.\]
 \end{enumerate}
\end{lem}
\begin{proof}
Statement (1) is well-known; see for example \cite[Proposition 1.4.6]{SH}. 

(2) follows from (1) by noticing that, since $NP(I)$ is a lattice polyhedron we have 
\[ NP(I)=\conv(NP(I)\cap \N^d)=\conv(\cL(\ov{I}))=NP(\ov{I}).\]

(3) %Suppose that $\bx^\ba\in G(\overline{I})$. 
If $\ba \in \np(I)$, the choice $\bb=\ba$ works as claimed. We may thus assume $\ba\not\in \np(I)$. 
By \eqref{eq:NP} there is $\by \in \np(I)$ such that the inequality $\ba\geq \by$ is satisfied coordinatewise. Since $\ba\in \N^d$, we have that $\ba\geq \lceil \by \rceil :=(\lceil y_1\rceil ,\ldots, \lceil y_d\rceil)$ and since $\lceil \by \rceil \geq \by$, we have $\lceil \by \rceil\in NP(I)$. As $\bx^\ba$ is a minimal generator of $\overline{I}$, it follows that $\ba= \lceil \by \rceil$. 

 Denote the unit hypercube in $\R^d$ by $H_d$; it has vertices $\sum_{i\in S\subseteq[d]}\be_i$. Since $x^\ba$ is a minimal generator of $\overline{I}$, it follows  that the only vertex of $\ba-H_d$ that is in $NP(I)$ is $\ba$. Moreover, since the only lattice points in $\ba-H_d$ are its vertices, the only lattice point in $(\ba-H_d)\cap NP(I)$ is $\ba$. Finally, we have  $\by\in \ba-H_d$ because $\ba=\lceil\by\rceil$.

%Since $\ba\not\in np(I)$ by assumption and the vertices of $\ba-H_d$ are the only lattice points in this hypercube, we have $np(I)\cap(\ba -H_d)\cap\N^d=\emptyset$. 
Let  $\bz \in \N^d$ be any vertex of $\np(I)$. From the previous considerations, we have $\bz\not\in \ba-H_d$.
Since $\np(I)$ is convex, the line segment $[\by, \bz]$ is contained in $\np(I)$. Let $\bb$ be the intersection point of this line segment with the boundary of the polytope $\ba-H_d$. Such an intersection point exists since $\by$ is inside and $\bz$ is outside $\ba-H_d$. Then $\bb \in \np(I)$ and $\lceil \bb\rceil$ is a vertex of $\ba-H_d$ that belongs to $NP(I)$; thus we have $\lceil \bb \rceil =\ba$. Furthermore, since $\bb\neq \ba-\mathbf{1}$, and $\bb$ is on the boundary of $\ba-H_d$, it follows that for some $1 \leq i \leq d$ we have $b_i = a_i$. Hence  we obtain $\sum_{i=1}^d (a_i-b_i) \leq d-1$, as claimed.

%$a_i-1<b_i$ for $1\leq i\leq d$, otherwise $\ba-\be_i\geq \bb$ which would yield $\bx^{\ba-\be_i}\in \overline{I}$ again by \eqref{eq:NP}. Summing up the preceding inequalities we obtain
%\[
%\sum_{i=1}^d (a_i-1)< \sum_{i=1}^d b_i \iff \sum_{i=1}^d (a_i-b_i)<d.
%\]
%Since $a_i, b_i\in \N$, the displayed inequality is equivalent to $\sum_{i=1}^d (a_i-b_i)\leq d-1$.
\end{proof}

\section{Real powers of monomial ideals}
\label{s:rp}

We now discuss powers of monomial ideals with real exponents, termed real powers, and their relationship to integral closure.

\begin{defn}
\label{def:realpower}
Fix a real number $r\geq 0$. We define the $r$-th {\em real power} of a monomial ideal, $I$,    to be 
\[\ov{I^r} =\left(\{\bx^\ba \mid  \ba \in r\cdot NP(I) \cap \N^d \}\right).\] When $r\in\Q$ we will refer to $\ov{I^r}$ as the $r$-th {\em rational power} of $I$.
\end{defn}

Rational powers of monomial ideals have appeared previously in the literature under the following definition and notation, see \cite[Definition 10.5.1]{SH}:  the $r$-th {\em rational power} of an arbitrary ideal $I$ of a ring $R$ for $r=\frac{p}{q}$ with $p,q\in\N, q\neq 0$ is the ideal 
\begin{equation}
\label{eq:altdef}
I_r:=\{y\in R \mid  y^q \in \ov{I^p}\},
\end{equation}
 where $\ov{I^p}$ denotes the integral closure of the $p$-th ordinary power of $I$, $I^p$. In the following we show that 
these two definitions agree, i.e., $I_r=\ov{I^r}$ whenever $r\in\Q$ and furthermore for natural exponents $r\in\N$ the $r$-th real power agrees with the integral closure of the $r$-th ordinary power of $I$, $I^r$. 

Our notation for real powers deviates from that in \eqref{eq:altdef}, which is more established in the literature, in favor of being intentionally consistent with the notation for integral closure, since these notions agree for $r\in\N$ as shown in the following lemma.

\begin{lem}
\label{lem:rational=intclosed}
Let $I$ be a monomial ideal. Then 
\begin{enumerate}
\item If $r\in \N$,  then the $r$-th real power of $I$ is equal to the integral closure of the $r$-th ordinary power $I^r$. In particular, the first rational power of $I$, $\overline{I^1}$, is the integral closure of $I$. Moreover, the $r$-th real power of $I$ is integrally closed.
\item  If $r\in \Q$  then the $r$-th real power of $I$ in \Cref{def:realpower} agrees with the  $r$-th rational power of $I$, $I_r$,  in  \eqref{eq:altdef}.
\end{enumerate}
\end{lem}
\begin{proof}
(1) By definition, a monomial $\mathbf{x}^\mathbf{a}$ is an element of the $r$-th real power of $I$ if and only if $\mathbf{a} \in r \cdot NP(I)$. Noting that $r\cdot NP(I) = NP(I^r)$ if $r\in\N$, the latter condition is equivalent to $\mathbf{a} \in NP(I^r)$. Now by \Cref{lem:intclosure} (1), we have $\mathbf{a} \in NP(I^r)\cap\N^d$ if and only if $\mathbf{x}^\mathbf{a}$ is an element of the integral closure of $I^r$ if and only if  $\mathbf{x}^\mathbf{a}$ is an element of the integral closure of $\overline{I^r}$.

(2) Let $r=\frac{p}{q}$ with $p,q\in\N, q\neq 0$ and let $\bx^\ba$ be a monomial. By \eqref{eq:altdef}, $\bx^\ba\in I_r$ holds if and only if we have $\bx^{q\ba}\in \ov{I^p}$,  equivalently $q\ba\in NP(\ov{I^p})=NP(I^p)=pNP(I)$. In turn, the last assertion is equivalent to  $\ba \in rNP(I)\cap\N^d$ and  by \Cref{def:realpower} this holds if and only if $\bx^\ba\in \ov{I^r}$. 
\end{proof}

Using \Cref{lem:intclosure}, for $r\in\Q_+$ we aim to confine the minimal generators of $\overline{I^r}$ to a bounded convex set, which will be obtained by Minkowski sum. In order to define this convex set we introduce the {\em unit simplex} in $d$-dimensional space,
\[
S_d=\{\ba=(a_1,\ldots, a_d)\in \R^d \mid a_1+\cdots+a_d\leq 1, a_i\geq 0 \text{ for } 1\leq i\leq d\}.
\]
In the metric space $\R^d$ endowed with the lattice distance, the unit simplex is the non negative portion of the ball of radius one centered at the origin. Denoting the origin in $\R^d$ by $\bf 0$, this observation yields an alternate description 
\[
S_d=\{\ba\in \R^d \mid \ba\geq \bf 0, \ \d(\ba,\bf 0)\leq 1\}.
\]

\begin{rem}
\label{rem:naturalboundeddistance}
\Cref{lem:intclosure} (3) can be reformulated using this notation as follows:  If $I$ is a monomial ideal and $\bx^\ba\in G(\overline{I})$, then $\ba \in \np(I)+(d-1)\cdot S_d$.
\end{rem}

The following technical result shall prove very useful for our purposes.

\begin{lem}
\label{lem:rationalmingen}
Let $\mathbf{x}^\mathbf{a}$ be a minimal generator of $\overline{I^r}$, where $r=\frac{p}{q}$ is a positive rational number. Then there exists a minimal generator $\mathbf{x}^\mathbf{b}$ of $\overline{I^p}$ such that  $q\ba-\bb \in d(q-1)\cdot S_d$.
\end{lem}
\begin{proof}
%To show $q\ba-\bb \in d(q-1)\cdot S_d$ we will prove the equivalent statements
%\[
%q\ba\geq \bb \text{ and } \sum_{i=1}^d(qa_i-b_i)\leq d(q-1).
%\]
% Suppose to the contrary that for all minimal generators, $\bx^\bb$, of $\overline{I^p}$ we have
%\[ \sum_{i=1}^d(qa_i-b_i)\geq d(q-1)+1.\] Applying the pigeon-hole principle, we find that there must exist  $i_0\in\{1,\ldots,d\}$ such that $qa_{i_0}-b_{i_0} \geq q$. Rewriting, we get that $q(a_{i_0}-1) \geq b_{i_0}$.

%Now let $\bx^\ba$ be a minimal generator of $\overline{I^r}$. 
By \Cref{lem:rational=intclosed} (2), we obtain $\bx^{q\ba} \in \overline{I^p}$. %, where $\overline{I^p}$ denotes simultaneously the integral closure of the $p$-th ordinary power of $I$ and the $p$-th real power of $I$, as shown in \Cref{lem:rational=intclosed} (1). 
Thus there exists a minimal generator  $\bx^\bb \in G(\overline{I^p})$ such that $\bx^\bb$ divides $\bx^{q\ba}$. This implies $\bb \leq q\ba$, that is, $b_i \leq qa_i$ for all $1 \leq i \leq d$. 
Suppose that $q\ba-\bb \not \in d(q-1)\cdot S_d$. Then  the inequality
\[
\sum_{i=1}^d(qa_i-b_i)\geq d(q-1)+1
\]
 follows by integrality. Applying the pigeon-hole principle, we find that there must exist  $i_0\in\{1,\ldots,d\}$ such that $qa_{i_0}-b_{i_0} \geq q$. Rewriting, we get that $q(a_{i_0}-1) \geq b_{i_0}$.
We can now set $\ba'=\ba-\be_{i_0}$ and with this notation we find
\[ \bb \leq q(a_1, \ldots, a_{i_{0-1}}, a_{i_0}-1, a_{i_{0+1}}, \ldots, a_d) =q\ba'.\] 
Thus $\bx^\bb$ divides $\bx^{q\ba'}$ and $\bx^{q\ba'}$ is an element of $\overline{I^p}$. Applying \Cref{lem:rational=intclosed} (2) again, this yields that, $\bx^{\ba'}\in \overline{I^r}$, which contradicts that $\bx^\ba$ is a minimal generator of $\overline{I^r}$.
\end{proof}

We are now able to describe a bounded convex set which contains the minimal generators of a rational power for a monomial ideal. The following result constitutes the basis for our Minkowski algorithm described in \Cref{alg:Minkowski}. See also \Cref{ex:Minkowski} for an illustration of the convex set $\mathcal{C}(I,r)$ defined below.

\begin{thm}
\label{thm:boundeddistance}
Let $I$ be a monomial ideal in $K[x_1,\ldots,x_d]$. If $r=\frac{p}{q}$ is a positive rational number and $\bx^\ba\in G(\overline{I^r})$, then $\ba$ is in the following bounded convex set
\begin{equation}
\label{eq:C}
\mathcal{C}(I,r)=r\cdot \np(I)+\left(d-\frac{1}{q}\right)\cdot S_d.
\end{equation}
Moreover, if $\ba\in \mathcal{C}(I,r)\cap\N^d$, then $\bx^\ba\in \ov{I^r}$ and thus $\ov{I^r}=(\{\bx^\ba\mid \ba\in \mathcal{C}(I,r)\cap\N^d)\}$.
\end{thm}
\begin{proof}
%The case where $q=1$ is covered in \Cref{lem:naturalboundeddistance}. 
By \Cref{lem:rationalmingen},  there exists a minimal generator of $\overline{I^p}$, $\mathbf{x}^\mathbf{b}$, such that 
\[q\ba-\bb \in d(q-1)\cdot S_d\] and from \Cref{rem:naturalboundeddistance} applied to the monomial ideal $I^p$ we have that 
\[\bb\in \np(I^p)+\left(d-1\right)\cdot S_d=p\cdot \np(I)+\left(d-1\right)\cdot S_d.\] 
Combining the displayed statements, we obtain
\begin{align*}
q\ba &\in p\cdot \np(I)+(d-1)\cdot S_d+ d(q-1)\cdot S_d \\
\iff \ba &\in \frac{p}{q}\cdot \np(I) + \frac{d-1}{q} \cdot S_d + \frac{d(q-1)}{q} \cdot S_d \\
\iff \ba &\in r\cdot \np(I)+\left(d-\frac{1}{q}\right)\cdot S_d.
\end{align*}
Finally, since $S_d\subseteq \R_+^d$, we have that $\mathcal{C}(I,r)\subseteq r\cdot NP(I)$ by \eqref{eq:NP=np+orthant}. Thus if $\ba\in \mathcal{C}(I,r)\cap\N^d$, then $\ba\in r\cdot NP(I)$ which yields  $\bx^\ba\in \ov{I^r}$ according to \Cref{def:realpower}. The identity $\ov{I^r}=(\bx^\ba\mid \ba\in \mathcal{C}(I,r)\cap\N^d)$ follows from the previous assertions.
\end{proof}

\begin{rem}
While the previous theorem does not require the rational number $r=\frac{p}{q}$ to have $\gcd(p,q)=1$, in applications is desirable to work with the reduced form of $r$ in order to obtain the smallest possible region $\mathcal{C}(I,r)$.
\end{rem}

\section{Algorithms for computing real powers}
\label{s:alg}

Several algorithms are proposed below for computing real powers of monomial ideals.

Our algorithms rely on several auxiliary computational tasks, which are highly non trivial, but can be performed currently by computer algebra systems such as  \cite{M2} or \cite{4ti2}. Specifically, we assume that independent routines are used to compute the Newton polyhedron or polytope for a given monomial ideal. For this reason, we take these convex bodies as input for our algorithms. For \Cref{alg:Minkowski} we additionally assume the existence of a routine that finds all the lattice points in a bounded convex polytope. This task is discussed in detail in \cite{LattE}. 

\subsection{Minkowski Algorithm} 

Our first algorithm uses the ideas presented in  \Cref{thm:boundeddistance} and illustrated in \Cref{ex:Minkowski} to confine the generators of a real power $\ov{I^r}$ within a convex region of bounded lattice distance from the Newton polytope $\np(I)$.

\smallskip

\begin{algorithm}[H]
\label{alg:Minkowski}
\DontPrintSemicolon
  \SetKw{KwTo}{in}
%   \SetKwFor{For}{for}{\string:}{}
  \KwInput{the Newton polytope $\np(I)$ of an ideal $I$, a rational number $r=\frac{p}{q}\in \Q_+$}
  \KwOutput{a list of monomial generators for the ideal $\overline{I^r}$}
  
  \tcc{Scaled newton polytope of I}
  
  scalednp := $r\cdot np(I)$
 
  \tcc{Bounded convex set, as given by \Cref{thm:boundeddistance}}
  
  $d := $ dimension of the polynomial ring containing $I$
  
  simplex := $d$-dimensional simplex with vertices at $\{\mathbf{0},(d-\frac{1}{q})\mathbf{e}_1, \ldots, (d-\frac{1}{q})\mathbf{e}_d\}$.
 
  $C$ := minkowskiSum(scalednp, simplex)
 
  \tcc{Find all lattice points and their monomial counterpart}
  
  exponentVectors := latticePoints($C$) \label{minklatpts}
  
  Initialize generators := $\emptyset$
  
  \For{$\bb$ \KwTo  exponentVectors}
    {
        generators := append($\bx^\bb$, generators)
    }
  \tcc{Return the possibly non minimal monomial generators}
   
  Return generators.
\caption{Minkowski Sum algorithm}
\end{algorithm}
\smallskip

\begin{prop}
If $I$ is a monomial ideal of a $d$-dimensional polynomial ring and $r\in\R_+$, then \Cref{alg:Minkowski} returns a not necessarily minimal set of monomial generators for $\overline{I^r}$.
\end{prop}

\begin{proof}
This follows from the assertion $\ov{I^r}=(\{\bx^\ba\mid \ba\in \mathcal{C}(I,r)\})\cap \N^d)$ of  \Cref{thm:boundeddistance}. In  \Cref{alg:Minkowski} the set $\mathcal{C}(I,r)$, termed $C$, is constructed according to  equation \eqref{eq:C}.
\end{proof}

\begin{ex}
\label{ex:Minkowski}

Consider the ideal $I=(xy^5, x^2y^2, x^4y)$ and the rational number $r=\frac{4}{3}$. Then one can determine that 
\[\ov{I^{4/3}}=(x^2y^5, x^2y^6, x^2y^7, x^3y^3, x^3y^4, x^3y^5, x^3y^6, x^4y^2, x^4y^3, \newline x^4y^4, x^4y^5, x^5y^2, x^5y^3, x^6y^2)\] based on identifying the lattice points in the convex region \[\mathcal{C}\left(I,\frac{4}{3}\right)=\frac{4}{3}\cdot \np(I)+\frac{5}{3}\cdot S_2\] given by \Cref{thm:boundeddistance}. Note that $\ov{I^{4/3}}$ is minimally generated by $G(\ov{I^{4/3}})=\{x^2y^5, x^3y^3, x^4y^2\}$. Thus, \Cref{alg:Minkowski} does not in general identify the minimal generators, but rather a possibly non minimal set of generators for $\ov{I^r}$. In the \Cref{fig:mink-alg-2}, the region $\mathcal{C}(I, \frac{4}{3})$ is shaded in darker blue, while the rest of the scaled polyhedron $\frac{4}{3}\cdot NP(I)$  is shaded in lighter blue.

\begin{figure}[h!]
    \centering \scalebox{0.75}{
    \begin{tikzpicture}
        % I = (xy^5, x^2y^2, x^4y), r = 4/3
        % d - 1/q = 2 - 1/3 = 5/3
        % See https://www.desmos.com/calculator/2kdpialxlh for double checking.
        
        \tikzstyle{every node}=[font=\small]
        \pgfplotsset{every axis legend/.append style={legend pos=outer north east,draw=NordWhite}}
        
        % % plots NOT in legend
        \begin{axis}  
        % grid size
        [posQuad, xtick={0,...,8}, ytick={0,...,9}, xmin=0, xmax=8, ymin=0, ymax=9]
        
        % minkowski sum region
        \addplot[fill=white!30!NordBrightBlue, opacity=0.5, line width=0pt] coordinates{(4/3,20/3) (8/3,8/3) (16/3,4/3) (16/3+5/3, 4/3) (4/3 + 5/3,20/3) (4/3,20/3+5/3)};
        \end{axis}
        
        %  % plots
        \begin{axis}
        [posQuad, xtick={0,...,8}, ytick={0,...,9}, xmin=0, xmax=8, ymin=0, ymax=9]
        
        % newton polyhedron region 
        \addplot[area style, fill=NordBrightBlue, opacity=0.5, line width = 2pt, draw = NordBrightBlue] coordinates{(4/3,9) (4/3,20/3) (8/3,8/3) (16/3,4/3) (8,4/3) (8,9)};
        \addlegendentry{$r \cdot NP(I)$}
        
        % dummy entry for minkowski sum region
        \addlegendimage{area style, fill=white!30!NordBrightBlue, line width=0pt, draw = NordBrightBlue}
        \addlegendentry{Convex Region}
    
        % minimal generators
        \addplot[only marks, mark=*, fill=NordMagenta, draw=NordMagenta]coordinates{(2,5) (3,3) (4,2)};
        \addlegendentry{Minimal Generators}
        
        % interior points (inside Minkowski sum)
        \addplot[only marks,mark=*, fill=NordYellow, draw=NordYellow] coordinates{(2,6) (2,7) (3,4) (3,5) (3,6)  (4,3) (4,4) (4,5)  (5,2) (5,3) (5,4) (6,2) };
        \addlegendentry{Interior Points}
        
        % interior points (outside Minkowski sum)
        % \addplot[only marks,mark=*, fill=NordYellow, draw=NordYellow] coordinates{ (2,8) (2,9) (3,7) (3,8) (3,9) (4,6) (4,7) (4,8) (4,9) (5,4) (5,5) (5,6) (5,7) (5,8) (5,9) (6,3) (6,4) (6,5) (6,6) (6,7) (6,8) (6,9) (7,2) (7,3) (7,4) (7,5) (7,6) (7,7) (7,8) (7,9) (8,2) (8,3) (8,4) (8,5) (8,6) (8,7) (8,8) (8,9) };
        \end{axis}

    \end{tikzpicture}}
    \caption{Computing $\ov{(xy^5, x^2y^2, x^4y)^{4/3}}$ using the Minkowski algorithm.}
    \label{fig:mink-alg-2}
\end{figure}
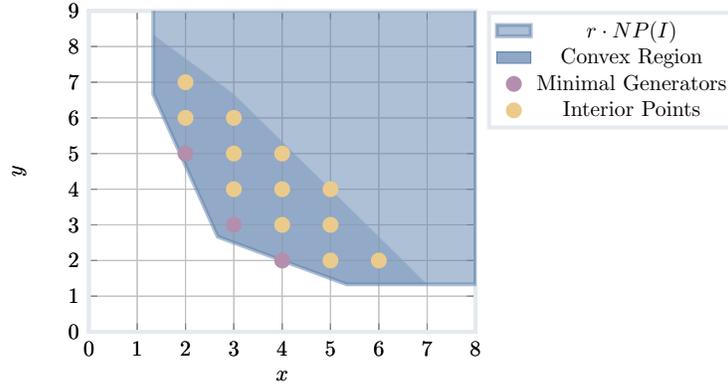

\end{ex}

\subsection{Hyperrectangle Algorithm}
The next algorithms depend on the notion of the hyperrectangle of a scaled Newton polyhedron, which is defined below.

\begin{defn} 
\label{defn:hyperrect}
Given a monomial ideal $I$ of a $d$-dimensional polynomial ring and $r\in\R_+$, define the set of scaled vertices of $I$ with respect to $r$ to be 
\[
\mathcal{V}(I,r) = \left\{\lceil r\ba \rceil := \left( \lceil ra_1 \rceil, \dots, \lceil ra_d \rceil \right) \, | \, x^\ba \in G(I) \right\}.
\] 
Let $\alpha = (\alpha_1, \dots , \alpha_d) \in   \mathcal{V}(I,r)$. Define 
\begin{equation}
\label{eq:minmax}
    \min(  \mathcal{V}(I,r), i) = \min\limits_{\alpha \in   \mathcal{V}(I,r)} \alpha_i \quad \text{ and } \quad \max(  \mathcal{V}(I,r), i) = \max\limits_{\alpha \in   \mathcal{V}(I,r)} \alpha_i.
\end{equation}
Finally, set the \emph{hyperrectangle of $r\cdot NP(I)$} to be the following set 
\begin{align*}
   % \label{eq:hype}
 \hype(I,r) &= \{ \bc = (c_1, \dots, c_d) \, | \, c_i \in [\min(  \mathcal{V}(I,r), i), \, \max(  \mathcal{V}(I,r), i)]  \}\\
 &=\prod_{i=1}^d  [\min(  \mathcal{V}(I,r), i), \, \max(  \mathcal{V}(I,r), i)].
\end{align*}
\end{defn}

We now see that the generators for the $r$-th real power of $I$ are among the set of lattice points in $\hype(I,r)$.

\begin{lem}
\label{lem:hype}
Let $I$ be a monomial ideal and let $r\in\R_+$. Denote the set of lattice points in $\hype(I,r)$  by $\mathcal{S}(I,r)$. Then 
\begin{enumerate}
\item $\lceil r\cdot \np(I) \rceil := \{(\lceil p_1 \rceil, \ldots, \lceil p_d\rceil) \mid \bp \in  r\cdot \np(I) \} \subseteq  \mathcal{S}(I,r)$
\item
$\ov{I^r}$ is generated by a subset of the lattice points in $\hype(I,r)$, more precisely
\[
\ov{I^r}=(\{\bx^\ba \mid \ba \in r\cdot NP(I)\cap \hype(I,r) \cap \N^d\}).
\]
\end{enumerate}
\end{lem}
\begin{proof}
(1)  Every point in $\bp\in r\cdot \np(I)$ is a convex combination of the vertices of this polytope, which are in the set $V=\{r\ba \mid \bx^\ba \in G(I)\}$. Since every coordinate $p_i$ of $\bp$ is a convex combination of $i$-th coordinates of elements in $V$ we obtain that $p_i\in [\min_{\ba\in V}{a_i}, \max_{\ba\in V}{a_i}]$ for $1\leq i\leq d$.  Thus $\lceil p_i \rceil \in [\min(  \mathcal{V}, i), \, \max(  \mathcal{V}, i)]$, which settles the claim.

(2) Temporarily denote $J:=(\bx^\ba \mid \ba \in r\cdot NP(I)\cap \mathcal{S}(I,r))$. Then $J\subseteq \ov{I^r}$ follows from \Cref{def:realpower}. Now let $\ba\in \N^d$ be such that $\bx^\ba\in \ov{I^r}$ and thus $\ba\in r\cdot NP(I)\cap \N^d$. From \eqref{eq:NP=np+orthant} we know
\[
 r\cdot NP(I) = r\cdot \np(I) +r\cdot \R_+^d= r\cdot \np(I) +\R_+^d,
\]
thus there exists $\bb\in r\cdot \np(I)$ such that $\ba\geq \bb$. Since $\ba\in \N^d$ it follows that $\ba\geq \lceil \bb\rceil=(\lceil b_1\rceil, \ldots, \lceil b_d\rceil]$, where $\lceil \bb\rceil \in \lceil r\cdot \np(I)\rceil$. From part (1) it follows that  $\lceil \bb\rceil \in \mathcal{S}(I,r)$ and from $\lceil \bb\rceil \geq \bb $ we deduce $\lceil \bb\rceil \in r\cdot \np(I)$ hence $\lceil \bb\rceil \in  r\cdot NP(I)+\R_+^d$. We have thus shown that  $\lceil \bb\rceil \in r\cdot NP(I)\cap \mathcal{S}(I,r)$, hence $\bx^{\lceil \bb\rceil} \in J$ and since $\ba\geq \lceil \bb\rceil $ we deduce $\bx^\ba\in J$. Thus the containment $\ov{I^r} \subseteq J$ has been established.
 \end{proof}

Based on the previous result we produce the following algorithm. 

\smallskip

\begin{algorithm}[H]
\label{alg:hyperrect}
\DontPrintSemicolon 
    \SetKw{KwTo}{in}
    % \SetKwFor{For}{for}{\string:}{}
    \KwInput{generators $G(I)$ and the Newton polyhedron $NP(I)$  of an ideal $I$, a real number $r\in\R_+$}
    \KwOutput{a list of monomial generators for the ideal $\overline{I^r}$}
    
    $d := $ dimension of the polynomial ring containing $I$
    
    candidates := $\hype(I,r)\cap \N^d$
    
    Initialize generators := $\emptyset$
    
    \For{$\bb $ \KwTo   candidates}
    {
        \If{$\bb $ \KwTo $r \cdot NP(I)$}{
        generators := append($\bx^\bb$, generators)
        }
    }
    
  Return generators.

\caption{Hyperrectangle algorithm}
\end{algorithm}

\begin{prop}
\label{prop:hyperrectalg}
If $I$ is a monomial ideal and $r\in\R_+$, then \Cref{alg:hyperrect} returns a not necessarily minimal set of monomial generators for $\overline{I^r}$.
\end{prop}

\begin{proof}
This follows from part (2) of \Cref{lem:hype}.
\end{proof}

\begin{ex}

\Cref{fig:hyperrect-alg-2} illustrates the set of lattice points in the hyperrectangle $\hype(I, \frac{4}{3})$ for the ideal $I=(xy^5, x^2y^2, x^4y)$. These are marked in solid yellow, solid purple and hollow black. The set of generators returned by \Cref{alg:hyperrect} corresponds to the yellow and purple lattice points, while the minimal generators correspond to the purple points.

    \begin{figure}[h!]
    \centering \scalebox{0.75}{
    %  I = (xy^5, x^2y^2, x^4y), r = 4/3
    \begin{tikzpicture}
        \tikzstyle{every node}=[font=\small]
        \pgfplotsset{every axis legend/.append style={legend pos=outer north east,draw=NordWhite}}
        
        %  % plots
        \begin{axis}
        [posQuad, xtick={0,...,8}, ytick={0,...,9}, 
        xmin=0, xmax=8, ymin=0, ymax=9]
        
        % newton polyhedron region 
        \addplot[area style, fill=NordBrightBlue, opacity=0.5, line width = 2pt, draw = NordBrightBlue] coordinates{(4/3,9) (4/3,20/3) (8/3,8/3) (16/3,4/3) (8,4/3) (8,9)};
        \closedcycle;
       % \addlegendentry{$r\cdot NP(I)$}
        
        % hyperrectangle boundary
        \addplot[NordOrange] coordinates{(2,2) (2,7) (6,7) (6,2) (2,2)};
       % \addlegendentry{Boundary of $\hype(I,r)$}
        
        % minimal generators
        \addplot[only marks, mark=*, fill=NordMagenta, draw=NordMagenta]coordinates{(2,5) (3,3) (4,2)};
       % \addlegendentry{Minimal Generators}
        
        % interior points 
        \addplot[only marks,mark=*, fill=NordYellow, draw=NordYellow] coordinates{(2,6) (2,7) (3,4) (3,5) (3,6) (3,7) (4,3) (4,4) (4,5) (4,6) (4,7) (5,2) (5,3) (5,4) (5,5) (5,6) (5,7) (6,2) (6,3) (6,4) (6,5) (6,6) (6,7)};
      %  \addlegendentry{Interior Points}
        
        % exterior points 
        \addplot[only marks,mark=*, fill=NordWhite, draw=NordBlack] coordinates{(2,2) (2,3) (2,4) (3,2)};
       %\addlegendentry{Exterior Points}
        \end{axis}

    \end{tikzpicture}

%  I = (xy^5, x^2y^2, x^4y), r = 4/3
\begin{tikzpicture}
    \tikzstyle{every node}=[font=\small]
    \pgfplotsset{every axis legend/.append style={legend pos=outer north east,draw=NordWhite}}
    
    %  % plots
    \begin{axis}
    [posQuad, xtick={0,...,8}, ytick={0,...,9}, 
    xmin=0, xmax=8, ymin=0, ymax=9]
    
    % newton polyhedron region 
    \addplot[area style, fill=NordBrightBlue, opacity=0.5, line width = 2pt, draw = NordBrightBlue] coordinates{(4/3,9) (4/3,20/3) (8/3,8/3) (16/3,4/3) (8,4/3) (8,9)};
    \closedcycle;
    \addlegendentry{$r\cdot NP(I)$}
    
    % hyperrectangle boundary
    \addplot[NordOrange] coordinates{(2,2) (2,7) (6,7) (6,2) (2,2)};
    \addlegendentry{Boundary of $\hype(I,r)$}
    
    % minimal generators
    \addplot[only marks, mark=*, fill=NordMagenta, draw=NordMagenta]coordinates{(2,5) (3,3) (4,2)};
    \addlegendentry{Minimal Generators}
    
    % interior points 
    \addplot[only marks,mark=*, fill=NordYellow, draw=NordYellow] coordinates{(5,2) (6,2)};
    \addlegendentry{Interior Points}
    
    % exterior points 
    \addplot[only marks,mark=*, fill=NordWhite, draw=NordBlack] coordinates{(2,2) (2,3) (2,4) (3,2)};
    \addlegendentry{Exterior Points}
    \end{axis}
\end{tikzpicture}}

    \caption{Computing $\ov{(xy^5, x^2y^2, x^4y)^{4/3}}$ using the Hyperrectangle algorithm (left) and Improved Hyperrectangle algorithm (right)}
    \label{fig:hyperrect-alg-2}
\end{figure}
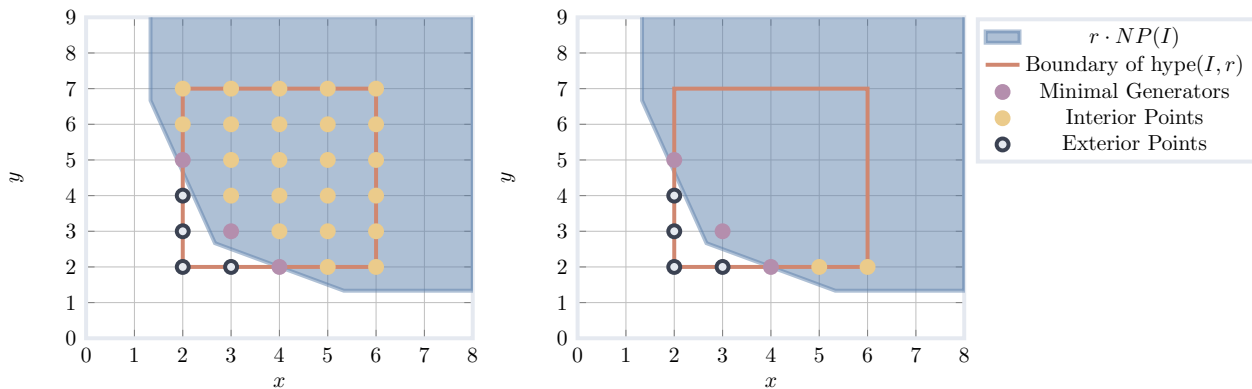
\end{ex}

In general, for fixed $I$ and $r$, the two convex sets $\mathcal{C}(I,r))$ and $\hype(I,r)$ where \Cref{alg:Minkowski} and \Cref{alg:hyperrect}, respectively, look for a set of generators for $\ov{I^r}$ are incomparable. For an illustration consider \Cref{fig:mink-alg-2} in \Cref{ex:Minkowski}, where the set $\mathcal{C}(I,r))$ is shaded in darker blue and \Cref{fig:hyperrect-alg-2} where the set $\hype(I,r)$ is the marked by the orange boundary. Note that there are no containments between the sets $\mathcal{C}(I,r)$ and $\hype(I,r)$ in this example. In general one does not expect a containment between the corresponding sets of lattice points $\mathcal{C}(I,r))\cap \N^d$ and $\hype(I,r)\cap \N^d$ either. However, the cardinality of the former set is typically smaller than the latter. We address this shortcoming in the next \Cref{alg:improvedhyperrect}. 

The exponent vectors for minimal generators of $\ov{I^r}$ are in $\mathcal{C}(I,r)\cap \hype(I,r)\cap \N^d$. However, as illustrated by \Cref{fig:mink-alg-2} and \Cref{fig:hyperrect-alg-2}, the exponents for the minimal generators of $\ov{I^r}$ can form a proper subset of $\mathcal{C}(I,r))\cap \hype(I,r)\cap\N^d$.

%The hyperrectangle algorithm is naive in the sense that it considers an extremely large number of lattice points in $\hype(I,r)\cap \N^d$ as candidates. 
The next variant improves on the hyperrectangle algorithm by reducing some redundancies in the traversal of lattice points. Using the while-loop on the final coordinate, the improved hyperrectangle algorithm stops looking for other generators after it finds a lattice point that is inside $r \cdot NP(I)$. Note that the improved hyperrectangle algorithm optimizes traversal of the set $\hype(I,r)\cap \N^d$ only on the last coordinate, so the benefits of using this algorithm over the hyperrectangle algorithm is more apparent in low dimensional rings.

\smallskip

\begin{algorithm}[H]
\label{alg:improvedhyperrect}
\DontPrintSemicolon 
    \SetKw{KwTo}{in}
    % \SetKwFor{For}{for}{\string:}{}
    \KwInput{the Newton polyhedron $NP(I)$ of an ideal $I$, a real number $r\in \R_+$}
    \KwOutput{a list of monomial generators for the ideal $\overline{I^r}$}
    $d := $ dimension of the polynomial ring containing $I$
    
    startPoints := $\{ \bb \,\in \hype(I,r) \; | \; b_d = \min(  \mathcal{V}, d)\}$
    
    Initialize generators := $\emptyset$
    
    \For{$\bb$ \KwTo startPoints}
    {
        \While{$\bb \ {\bf not}$ \KwTo $ r \cdot NP(I) \ {\bf and }\ b_d \leq \max(  \mathcal{V}, d)$}{
        $\bb := \bb + (0, \dots, 0, 1)$ \qquad \tcc{ ``move up''}
        }
        
        \If{$\bb$ \KwTo $r \cdot NP(I)$}{
        generators := append($\bx^\bb$, generators)
        }
    }
    \tcc{Return the possibly non minimal monomial generators}
    
    Return  generators.
\caption{Improved Hyperrectangle algorithm}
\end{algorithm}

%\smallskip

\begin{ex}

\Cref{fig:hyperrect-alg-2} illustrates the set of generators for the ideal $\ov{(xy^5, x^2y^2, x^4y)^{4/3}}$ returned by the improved hyperrectangle algorithm. The set of lattice points con\-si\-dered by this algorithm are marked in solid yellow and purple and hollow black. The set of generators returned by \Cref{alg:improvedhyperrect} corresponds to the yellow and purple lattice points, while the minimal generator correspond to the purple lattice points only. Compared to \Cref{fig:hyperrect-alg-2}, fewer non minimal generators are returned.
\end{ex}

%\begin{figure}[h!]
%\centering \scalebox{0.75}{
%%  I = (xy^5, x^2y^2, x^4y), r = 4/3
%\begin{tikzpicture}
%    \tikzstyle{every node}=[font=\small]
%    \pgfplotsset{every axis legend/.append style={legend pos=outer north east,draw=NordWhite}}
%    
%    %  % plots
%    \begin{axis}
%    [posQuad, xtick={0,...,8}, ytick={0,...,9}, 
%    xmin=0, xmax=8, ymin=0, ymax=9]
%    
%    % newton polyhedron region 
%    \addplot[area style, fill=NordBrightBlue, opacity=0.5, line width = 2pt, draw = NordBrightBlue] coordinates{(4/3,9) (4/3,20/3) (8/3,8/3) (16/3,4/3) (8,4/3) (8,9)};
%    \closedcycle;
%    \addlegendentry{$r\cdot NP(I)$}
%    
%    % hyperrectangle boundary
%    \addplot[NordOrange] coordinates{(2,2) (2,7) (6,7) (6,2) (2,2)};
%    \addlegendentry{Boundary of $\hype(I,r)$}
%    
%    % minimal generators
%    \addplot[only marks, mark=*, fill=NordMagenta, draw=NordMagenta]coordinates{(2,5) (3,3) (4,2)};
%    \addlegendentry{Minimal Generators}
%    
%    % interior points 
%    \addplot[only marks,mark=*, fill=NordYellow, draw=NordYellow] coordinates{(5,2) (6,2)};
%    \addlegendentry{Interior Points}
%    
%    % exterior points 
%    \addplot[only marks,mark=*, fill=NordWhite, draw=NordBlack] coordinates{(2,2) (2,3) (2,4) (3,2)};
%    \addlegendentry{Exterior Points}
%    \end{axis}
%\end{tikzpicture}}
%
%\caption{Computing $\ov{(xy^5, x^2y^2, x^4y)^{4/3}}$ using the \Cref{alg:improvedhyperrect}}
%\label{fig:improv-hyperrect-alg-2}
%\end{figure}
%% \end{comment}

\subsection{Staircase Algorithm}

 The algorithms presented in the previous sections (\Cref{alg:Minkowski}, \Cref{alg:hyperrect}, and \Cref{alg:improvedhyperrect}) have one common disadvantage in that they return possibly {\em non minimal} sets of generators for the real powers of monomial ideals. The next algorithm, termed the staircase algorithm,  traverses lattice points near the boundary of the Newton polyhedron. The traversal is designed so that, in the 2-dimensional case,  the minimal generators are  found. 
 
A benefit of the following algorithm is to improve upon the runtime of \Cref{alg:Minkowski} and \Cref{alg:improvedhyperrect}.  \Cref{alg:Minkowski} is slow in practice because of lattice points identification in step \ref{minklatpts}, while \Cref{alg:improvedhyperrect} may be inefficient because  a large number of operations could be performed to check if lattice points are in or outside $r \cdot NP(I)$. %Thus, the previous algorithms were found to be less computationally ideal for monomial ideals with large number of exponents or variables. 
To alleviate this issue, the staircase algorithm optimizes the traversal of lattice points on the final two coordinates. The algorithm uses the notation in equation \eqref{eq:minmax}.

\smallskip

\begin{algorithm}[H]
\label{alg:staircasealg}
\DontPrintSemicolon 
    \SetKw{KwTo}{in}
    \KwInput{the Newton polyhedron $NP(I)$ of an ideal $I$, a real number $r\in \R_+$}
    \KwOutput{a list of monomial generators for the real power $\overline{I^r}$}
    Initialize generators := $\emptyset$
    
    $d := $ dimension of the polynomial ring containing $I$
     
      \If{$d=1$}
      {Return $\{\bx^{\min(\mathcal{V},1)} \}$}
     \Else{
    startPoints := $\big\{ \ba \in \hype(I,r) \; | \;  a_{d-1} = \min(\mathcal{V},d-1), \; a_{d} = \max(\mathcal{V},d) \big\}$ \label{startpt}
    
    \For{$\ba$ \KwTo startPoints}{
        $\bb := \ba$
        
        \While{$\ba$ \KwTo $\hype(I,r)$}{  \label{whileinhype}
            \If{$\ba$ \KwTo $r \cdot NP(I)$}
            {
                $\bb := \ba$
                
                $\ba := \ba - (0, \dots, 0, 1)$ \qquad \tcc{``move down'' \label{movedown} }                
            }
            \Else{
            	\If{$\bb$ \KwTo $r \cdot NP(I)$  \label{NPcond1} }
                {
                    generators := append($\bx^\bb$, generators) \label{add1}
                }
                $\bb := \ba$
                
                $\ba := \ba + (0, \dots, 1,0)$  \qquad \tcc{``move right''}
            }
        }
        
        \If{$\bb $ \KwTo $r \cdot NP(I)$ \label{NPcond2} }
        {
            generators := append($\bx^\bb$, generators) \label{add2}
        } 
    }
    }
    
 Return generators.
\caption{Staircase algorithm}
\end{algorithm}

\begin{ex}
\Cref{fig:stair-alg-2} shows the set of lattice points considered by the staircase algorithm within $\hype(I, \frac{4}{3})$ for the ideal $I=(xy^5, x^2y^2, x^4y)$. While all the lattice points along the path of the algorithm are considered, only the minimal generators corresponding to the purple lattice points are returned.

\begin{figure}[h!]
    \centering \scalebox{0.75}{
    %  I = (xy^5, x^2y^2, x^4y), r = 4/3
    \begin{tikzpicture}
        \tikzstyle{every node}=[font=\small]
        \pgfplotsset{every axis legend/.append style={legend pos=outer north east,draw=NordWhite}}
        
        %  % plots
        \begin{axis}
        [posQuad, xtick={0,...,8}, ytick={0,...,9}, 
        xmin=0, xmax=8, ymin=0, ymax=9]
        
        % newton polyhedron region 
        \addplot[area style, fill=NordBrightBlue, opacity=0.5, line width = 2pt, draw = NordBrightBlue] coordinates{(4/3,9) (4/3,20/3) (8/3,8/3) (16/3,4/3) (8,4/3) (8,9)};
        \closedcycle;
        \addlegendentry{$r\cdot NP(I)$}
        
        % hyperrectangle boundary
        \addplot[NordOrange] coordinates{(2,2) (2,7) (6,7) (6,2) (2,2)};
        \addlegendentry{Boundary of $\hype(I,r)$}
        
        % path by the algorithm
        \addplot[NordBrightBlack]coordinates{(2,7) (2,6) (2,5) (2,4) (3,4) (3,3) (3,2) (4,2)};
        \addlegendentry{Path of algorithm}
    
        % minimal generators
        \addplot[only marks, mark=*, fill=NordMagenta, draw=NordMagenta]coordinates{(2,5) (3,3) (4,2)};
        \addlegendentry{Minimal Generators}
        
        % interior points on algorithm's path
        \addplot[only marks,mark=*, fill=NordYellow, draw=NordYellow] coordinates{(2,7) (2,6) (3,4)};
        \addlegendentry{Interior Points}
        
        % exterior points on algorithm's path
        \addplot[only marks,mark=*, fill=NordWhite, draw=NordBlack] coordinates{(2,4) (3,2)};
        
        % start and end labels
        \node[anchor=south, fill = NordWhite] at (2,7) {start};
        \node[anchor=north west, fill = NordWhite] at (4,2) {end};
        \addlegendentry{Exterior Points}
        \end{axis}
    \end{tikzpicture}}
    
    \caption{Computing $\ov{(xy^5, x^2y^2, x^4y)^{4/3}}$ using the Staircase algorithm}
    \label{fig:stair-alg-2}
\end{figure}
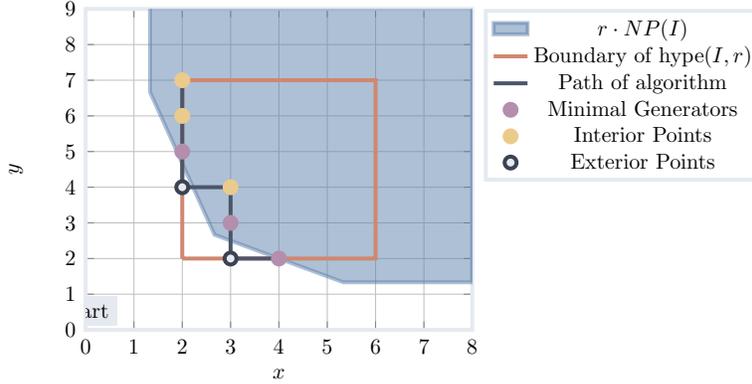
\end{ex}

We are now ready to show the validity of \Cref{alg:staircasealg}. We utilize terminology that is consistent with the visual descriptions in \Cref{fig:stair-alg-2}. We call the {\em path} of the algorithm $\mathcal{P}(I,r)$ the set of values taken by the variable $\ba$ in \Cref{alg:staircasealg} for fixed inputs $I,r$. This set is the disjoint union of two subsets: the exterior path and the interior path defined below:
\begin{align*}
\mathcal{P}_{ext}(I,r) &=\{\ba \in \mathcal{P}(I,r) \setminus r\cdot NP(I)\}\\
\mathcal{P}_{int}(I,r) &=\{\ba \in \mathcal{P}(I,r) \cap r\cdot NP(I)\}.
 \end{align*}
 \begin{prop}
\label{thm:staircasealg}
If $I$ is a monomial ideal of a $d$-dimensional polynomial ring and 
 $d\in\{1,2\}$, then \Cref{alg:staircasealg} returns a minimal set of monomial generators for $\overline{I^r}$.
If $d\geq 3$ then \Cref{alg:staircasealg} returns a not necessarily minimal set of monomial generators for $\overline{I^r}$. 
\end{prop}

\begin{proof}
In the case $d=1$, every monomial ideal $J\subseteq K[x_1]$ is principal, minimally generated by $x_1^m$, where $m=\min\{a\mid x_1^a\in J\}$. Applying this to $J=\ov{I^r}$ for which case $m=\min(\mathcal{V},1)$ yields $G(\ov{I^r})=\{x_1^{\min(\mathcal{V},1)}\}$, i.e., the output of \Cref{alg:staircasealg} in step 4.

For the case $d=2$, first notice that because of the succession of down moves and right moves, the interior path $\mathcal{P}_{int}(I,r)$
is a disjoint union of vertical strips of the form \[s_{a,b,c}:=\{\gamma=(\gamma_1,\gamma_2) \mid \gamma_1=a, \, \gamma_2\in [b,c]\cap \N\},\]
where $b=\min\{b'\mid (a,b')\in r\cdot NP(I)\}$ by step \ref{movedown} of the algorithm; see \Cref{fig:stair-alg-2} for an illustration. Moreover, the interior path contains one lattice point for each value of the $x_2$-coordinate in $[\min(\mathcal{V},1),\max(\mathcal{V},1)]$ so that in the decomposition 
\begin{equation}
    \label{eq:strips}
\mathcal{P}_{int}(I,r)=\bigcup_{i=\min(\mathcal{V},1)}^e s_{i,b_i,c_i}
\end{equation}
we must have $c_{\min(\mathcal{V},1)}= \max(\mathcal{V},2)$ and  $b_i=c_{i+1}+1$ for each $i\leq e-1$, where $e$ is the maximum $x_1$ coordinate of any point on the interior path. In particular, if $i<j$ then the inequality $b_i>c_j$ holds.

%Note that by definition $\bx^\ba\in\ov{I^r}$ for each $\ba\in \mathcal{P}_{int}(I,r)$.
Let $\bx^\ba\in G(\ov{I^r})$.  By \Cref{lem:hype} it follows that $\ba=(a_1,a_2) \in \hype(I,r)$, so $a_2\in [\min(\mathcal{V},1),\max(\mathcal{V},1)]$, and by the preceding remarks there exists a unique point $\bb\in \mathcal{P}_{int}(I,r)$ with $b_2=a_2$.  We claim that $\bb=\ba$. If not, then $a_1<b_1$ since $\bx^\ba$ is a minimal generator (i.e., $\ba$ lies ``left'' of $\bb$), and for this reason $a_2 = b_2\leq c_{b_1}< b_{a_1}$ (i.e $\ba$ lies ``below'' the strip with $x_1$-coordinate $a_1$). Since $\ba=(a_1,a_2)\in r\cdot NP(I)$ and $b_{a_1}=\min\{b'\mid (a_1,b')\in r\cdot NP(I)\}$, this yields a contradiction. We have shown that 
\[G(\ov{I^r})\subseteq \{\bx^\ba \mid \ba \in\mathcal{P}_{int}(I,r)\}.\]

In the notation of \eqref{eq:strips}, the algorithm returns the set $\{x_1^ix_2^{b_i}\mid \min(\mathcal{V},1)\leq i\leq e\}$.
Each of the  monomials $x_1^ix_2^j$ with $j\in(b_i,c_i]\cap \N$ are not in $G(\ov{I^r})$ since they are divisible by $x_1^ix_2^{b_i}$. Thus $G(\ov{I^r})$ is contained in the returned set. Moreover, the returned set consists of minimal generators since no two of its elements are comparable under the divisibility relation. In fact, this proof shows that the case $d=2$ of the algorithm gives a minimal set of generators for the ideal generated by the monomials with exponents in a given convex set (in our application to real powers, this convex set is $r\cdot NP(I)$). We use this to approach the higher dimensional cases. 

The case $d>2$ is derived from the case $d=2$ by the following analysis. By virtue of \Cref{lem:hype} we have the identity
\begin{align*}
\ov{I^r} &=\left( \{\bx^\ba \mid \ \ba\in \hype(I,r)\cap r\cdot NP(I)\cap \N^d\right)\\
&= \left( \sum_{\gamma\in \prod_{i=1}^{d-2}  [\min(  \mathcal{V}, i), \, \max(  \mathcal{V}, i)]} x_1^{\gamma_1}\cdots x_{d-2}^{\gamma_{d-2}} \cdot I_{\gamma,r}\right),
\end{align*}
where  $I_{\gamma,r}:=(\{x_{d-1}^ax_d^b\mid \ (\gamma_1, \ldots, \gamma_{d-2}, a,b)\in r\cdot NP(I) \})$ is an ideal in a 2-dimensional polynomial ring. 
According to the case $d=2$, steps 7--19 of the algorithm append the set $ x_1^{\gamma_1}\cdots x_{d-2}^{\gamma_{d-2}}\cdot G(I_{\gamma,r})$ to the generators list. The union of these sets generates $\ov{I^r}$ by the above displayed identity.
\end{proof}

\begin{ex}
We give a visual illustration of using \Cref{alg:staircasealg} to compute the integral closure of $I = (y^3, y^2z^5, x^2y^2, x^2z^3)$, that is, $\ov{I^1}$ in \Cref{fig:3dstaircase}. In 3-dimensional space, the path of the algorithm is a disjoint union of paths, each corresponding to an ideal in a 2-dimensional ring as shown in the proof of \Cref{thm:staircasealg}.

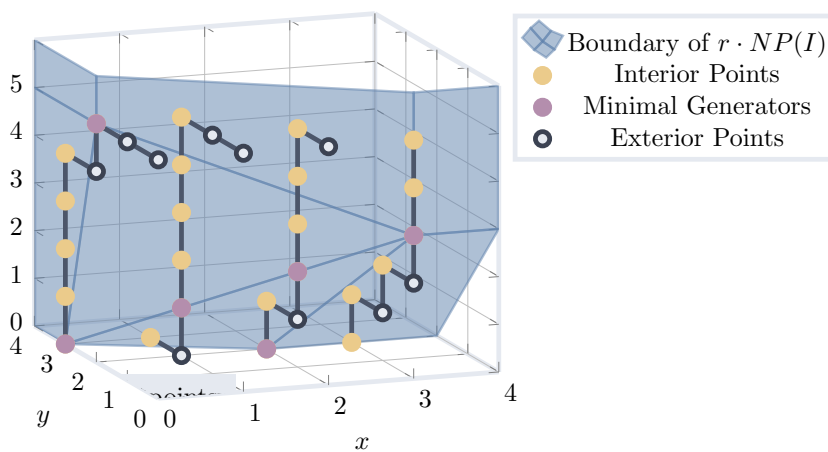
\begin{figure}[h!]
\centering \scalebox{0.9}{
\begin{tikzpicture}
    \tikzstyle{every node}=[font=\small]
    \pgfplotsset{every axis legend/.append style={legend pos=outer north east,draw=NordWhite}}
    
    \begin{axis}[view={-20}{15},posQuad, xtick={0,...,4}, ytick={0,...,4}, ztick={0,...,5},
    xmin=0, xmax=4, ymin=0, ymax=4, zmin=0, zmax=6]
    
    % newton polyhedron
    \addplot3[surf, faceted color = NordBrightBlue, color = NordBrightBlue, opacity = 0.5, line width = 1pt] coordinates {
     (2,2,0) (3,0,3) (0,2,5)
			 
	 (2,2,0) (0,3,0) (0,2,5)
	};
	\addlegendentry{Boundary of $r \cdot NP(I)$}
	
    % hyperrectangle boundary
    % \addplot3[NordOrange] coordinates{(0,0,0) (0,4,0) (0,4,5) (4,4,5) (4,0,5) (0,0,5) (0,0,0) (4,0,0) (4,0,5) (4,4,5) (4,4,0) (0,4,0) (0,4,5) (0,0,5) (4,0,5) (4,0,0)};
    % \addlegendentry{Boundary of $\hype(I,r)$}
        
    % interior points on algorithm's path
    \addplot3[only marks, mark=*, fill=NordYellow, draw=NordYellow] coordinates{(0,2,5) (0,3,4) (0,3,3) (0,3,2) (0,3,1) (0,3,0) (1,2,5) (1,2,4) (1,2,3) (1,2,2) (1,2,1) (1,3,0) (2,1,5) (2,1,4) (2,1,3) (2,1,2) (2,2,1) (2,2,0) (3,0,5) (3,0,4) (3,0,3) (3,1,2) (3,2,1) (3,2,0) };
    \addlegendentry{Interior Points}
    
    % minimal generators
    \addplot3[only marks, mark=*, fill=NordMagenta, draw=NordMagenta]coordinates{(2,2,0) (3,0,3) (0,2,5) (0,3,0) (1,2,1) (2,1,2)};
    \addlegendentry{Minimal Generators}

    % exterior points on algorithm's path
    \addplot3[only marks, mark=*, fill=NordWhite, draw=NordBlack] coordinates{(0,0,5) (0,1,5) (0,2,4) (1,0,5) (1,1,5) (1,2,0) (2,0,5) (2,1,1) (3,0,2) (3,1,1)};
    \addlegendentry{Exterior Points}
    
    % left r NP(I) face
    \addplot3[surf, faceted color = NordBrightBlue, color = NordBrightBlue, opacity = 0.5, line width = 1pt] coordinates {
     (0,3,0) (0,2,5) (0,2,6)
			 
	 (0,4,0) (0,4,5) (0,4,6)
	};
    % top r NP(I) face	
	\addplot3[surf, faceted color = NordBrightBlue, color = NordBrightBlue, opacity = 0.5, line width = 1pt] coordinates {
	 (0,2,6) (3,0,6) (4,0,6)
	 
	 (0,2,5) (3,0,3) (4,0,3)
	};
    % bottom r NP(I) right face	
	\addplot3[surf, faceted color = NordBrightBlue, color = NordBrightBlue, opacity = 0.5, line width = 1pt] coordinates {
     (3,0,3) (4,0,3)
			 
	 (2,2,0) (4,2,0)
	};
    
    % x = 0 staircase traversal
	\addplot3[NordBrightBlack]coordinates {(0,0,5)  (0,1,5) (0,2,5) (0,2,4) (0,3,4) (0,3,3) (0,3,2) (0,3,1) (0,3,0) };
	
	% x = 1 staircase traversal
	\addplot3[NordBrightBlack]coordinates {(1,0,5) (1,1,5) (1,2,5) (1,2,4) (1,2,3)  (1,2,2) (1,2,1) (1,2,0) (1,3,0) };
	
	% x = 2 staircase traversal
	\addplot3[NordBrightBlack]coordinates {(2,0,5) (2,1,5) (2,1,4) (2,1,3) (2,1,2)  (2,1,1) (2,2,1) (2,2,0)};
	
	% x = 3 staircase traversal
    \addplot3[NordBrightBlack]coordinates {(3,0,5) (3,0,4) (3,0,3) (3,0,2) (3,1,2)  (3,1,1) (3,2,1) (3,2,0)};
    
    % start and end point labels
    \node[fill = NordWhite] at (4,3,4.7) {start points};
    \node[fill = NordWhite] at (0.7,0.7,0) {end points};
    \end{axis}
    \end{tikzpicture}}
     \caption{Computing generators for $\ov{(y^3, y^2z^5, x^2y^2, x^2z^3)}$ using the Staircase algorithm}
    \label{fig:3dstaircase}
    \end{figure}
\end{ex}

\section{Continuity and jumping numbers for exponentiation}
\label{s:jumping}

In this section we analyze how the real powers of monomial ideals vary with the exponent. To be precise, for a fixed monomial ideal $I$ we consider continuity properties for the {\em exponentiation function} of base $I$
\begin{equation*}
    \exp:\R_+\to \mathcal{T}, \quad \exp(r)=\ov{I^r}    
\end{equation*}
whose domain is $\R_+$ with its Euclidean topology and whose codomain is the set $\mathcal{T}=\{\ov{I^r}\mid r\in \R_+\}$  endowed with the discrete topology.

We start with two elementary properties enjoyed by the family of real powers of the fixed ideal.

\begin{lem}
\label{lem:gradedfam}
If $I$ is a monomial ideal and $r,s \in \R_+$ then
\begin{enumerate}
    \item  if $s\geq r\geq 0 $, then the containment $\ov{I^s}\subseteq \ov{I^r}$ holds,
    \item $\ov{I^s}\cdot \ov{I^r}\subseteq \ov{I^{s+r}}$.
\end{enumerate}
\end{lem}
\begin{proof}
Assertion (1) is clear from \Cref{def:realpower}. To clarify assertion (2), note that monomials in $\ov{I^s}\cdot \ov{I^r}$ correspond to lattice points in the Minkowski sum 
\begin{equation*}
    s\cdot NP(I)+r\cdot NP(I)=(s+r)\cdot NP(I).
\end{equation*}
\end{proof}

Part (2) of \Cref{lem:gradedfam} shows that the real powers of a fixed monomial ideal form a {\em graded family}, although this terminology is more commonly used for families indexed by a discrete set. 
Property (1) of \Cref{lem:gradedfam} allows to define for each $r\in\R$ the monomial ideal
\begin{equation*}
    \ov{I^{>r}}=\bigcup_{s>r} \ov{I^s}.
\end{equation*}
We show that this ideal can be understood as a limit in $\mathcal{T}$, meaning that a sequence of real powers of $I$ where the exponents approach a real number $r$ from the right must stabilize to $\ov{I^{>r}}$.

\begin{prop}
\label{lem:I>r=I^t_n}
Let $I$ be a monomial ideal and let $\{t_n\}_{n\in\N}$ be a non-increasing sequence of non-negative real numbers with $r=\lim_{n\to \infty} t_n$. Then $\ov{I^{ t_n}}=\ov{I^{>r}}$ for $n$ sufficiently large. 
\end{prop}
\begin{proof}
A non-increasing sequence of non-negative numbers $\{t_n\}_{n\in\N}$ gives an ascending chain of ideals $\ov{I^{t_0}}\subseteq \ov{I^{ t_1}}\subseteq \cdots \subseteq \ov{I^r}$ by \Cref{lem:gradedfam} (1). Since the polynomial ring is Noetherian, any such chain must in fact stabilize, i.e. there exists $N \gg 0$ such that $\ov{I^{ t_n}}=\ov{I^{t_m}}$ for $m,n\geq N$. We show that the stable value of this chain is $\ov{I^{>r}}$. Indeed, from the definition of $\ov{I^{>r}}$ one deduces the containment
\[
\ov{I^{t_N}}=\bigcup_{n=0}^\infty \ov{I^{t_n}} \subseteq \bigcup_{s>r} \ov{I^s}=\ov{I^{>r}} .
\]
Conversely, for each $s>r$, there exists $n\geq N$ such that $s>t_n$, hence one has the containments $\ov{I^s} \subseteq \ov{I^{t_N}}=\ov{I^{t_n}}$ for all $s>r$ and consequently  $\ov{I^{t_N}}\supseteq \ov{I^{>r}}$.
\end{proof}

To distinguish those real numbers $r$ for which the function $\exp:\R_+\to \mathcal{T}, \exp(r)= \ov{I^r}$ is right discontinuous, we term them jumping numbers.

 \begin{defn}
 \label{defn:jumping}
A {\em jumping number} for $I$ is a real number $r\in \R_+$ for which the real powers of $I$ are not equal to $\ov{I^r}$ when we approach $r$ from the right, i.e.
\begin{equation*}
    \ov{I^r}\neq \ov{I^{>r}}.    
\end{equation*}
\end{defn}

\begin{ex}
0 is a jumping number for any monomial ideal since $\ov{I^0}=R$ but $\ov{I^r}$ is a proper ideal for any $r>0$.
\end{ex}

\begin{ex}
For $I=(x^4,x^2y,xy^3)$ we have that $\frac{1}{3}$ is not a jumping number
while $\frac{1}{2}$ is a jumping number. This is because for small values of $\varepsilon>0$ there is an equality 
\[\frac{1}{3}\cdot NP(I)\cap\N^2=\left(\frac{1}{3}+\varepsilon \right)\cdot NP(I)\cap \N^2,\]
while 
\[\frac{1}{2}\cdot NP(I)\cap\N^2 \neq\left(\frac{1}{2}+\varepsilon \right)\cdot NP(I)\cap \N^2\]
because the point $(2,0)$ belongs to the leftmost set but not the rightmost. In fact, for the ideal $I$ in this example, we have $(x^2, xy)=\ov{I^{1/3}}=\ov{I^{>1/3}}=\ov{I^{1/2}}\neq \ov{I^{>1/2}}=(x^3,xy)$.

\begin{figure}[!h]
 \centering \scalebox{0.6}{
\begin{tikzpicture}
    \pgfplotsset{every axis legend/.append style={legend pos=outer north east,fill=NordBlack,draw=NordWhite}}
    \begin{axis}  
    [posQuad, xtick={0,...,3}, ytick={0,...,3}, 
    xmin=0, xmax=3, ymin=0, ymax=3]
    \addplot[NordCyan, mark=*, mark color=NordCyan] coordinates{(1.33,0) (.66,0.33) (0.33,1)};
    \addplot[only marks,mark=*, fill=NordYellow, draw=NordYellow] coordinates{(2,2) (3,2) (3,3) (2,3) (3,1) (3,0) (2,1) (1,2) (1,3)};
    \addplot[only marks, mark=*, fill=NordYellow, draw=NordYellow]coordinates{(1,1) (2,0)};
    \addplot[NordCyan] coordinates{(5,0) (1.33,0) (.66,0.33) (0.33,1) (0.33,5)};
    \addplot[fill=NordBrightBlue, opacity=0.5, line width=0pt] coordinates{(5,0) (1.33,0) (.66,0.33) (0.33,1) (0.33,5) (5,5)};
    \addplot[only marks,mark=*, fill=NordYellow, draw=NordYellow] coordinates{(2,2) (3,2) (3,3) (2,3) (3,1) (3,0) (2,1) (1,2) (1,3)};
    \addplot[only marks, mark=*, fill=NordYellow, draw=NordYellow]coordinates{(1,1) (2,0)};
\end{axis}\end{tikzpicture}

\begin{tikzpicture}
    \pgfplotsset{every axis legend/.append style={legend pos=outer north east,fill=NordBlack,draw=NordWhite}}
    \begin{axis}  
    [posQuad, xtick={0,...,3}, ytick={0,...,3}, 
    xmin=0, xmax=3, ymin=0, ymax=3]
    \addplot[NordCyan, mark=*, mark color=NordCyan] coordinates{(2,0) (1,0.5) (0.5,1.5)};
    \addplot[only marks,mark=*, fill=NordYellow, draw=NordYellow] coordinates{(2,2) (3,2) (3,3) (2,3) (3,1) (3,0) (2,1) (1,2) (1,3)};
    \addplot[only marks, mark=*, fill=NordYellow, draw=NordYellow]coordinates{(1,1) (2,0)};
    \addplot[NordCyan] coordinates{(5,0) (2,0) (1,0.5) (0.5,1.5) (0.5,5)};
    \addplot[fill=NordBrightBlue, opacity=0.5, line width=0pt] coordinates{(5,0) (2,0) (1,0.5) (0.5,1.5) (0.5,5) (5,5)};
    \addplot[only marks,mark=*, fill=NordYellow, draw=NordYellow] coordinates{(2,2) (3,2) (3,3) (2,3) (3,1) (3,0) (2,1) (1,2) (1,3)};
    \addplot[only marks, mark=*, fill=NordYellow, draw=NordYellow]coordinates{(1,1) (2,0)};
\end{axis}\end{tikzpicture}
}
\caption{Comparing $\frac{1}{3} \cdot NP(I)$ and $\frac{1}{2} \cdot NP(I)$}
\end{figure}
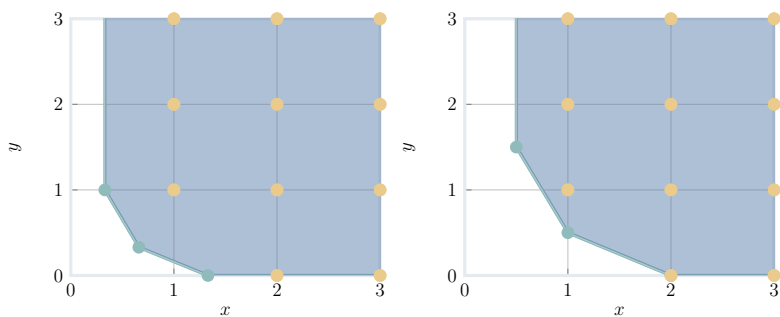
\end{ex}

To verify that right continuity is a special characteristic to study, we show that the exponentiation function  is a left continuous  function.

Towards this end recall that any polyhedron admits a description as a finite intersection of half spaces. We term the linear inequalities describing a polyhedron as an intersection of half spaces its bounding inequalities. In particular, if $I$ is a monomial ideal in a polynomial ring of dimension $d$ then $NP(I)$ is a lattice polyhedron, hence there exist a $d\times s$ matrix $A$ with entries in $\N$  and a vector $\bc\in\N^d$ such that
\begin{equation}
\label{eq:A}
 NP(I) = \{\bx \in \R_+^d \mid  A\bx \geq {\bc}\}.
\end{equation}
In \eqref{eq:A}, if  $A=[a_{ij}]$, we will further assume that  we have $\gcd(a_{i1},\ldots, a_{is},c_i)=1$ for each $1\leq i\leq d$.
 Moreover, scaling the Newton polyhedron amounts to scaling the constant term of the bounding inequalities,
that is,
\[
r\cdot NP(I) = \{\bx \in \R_+^d \mid  A\bx \geq r\cdot {\bc}\}.
\]

\begin{prop}
\label{lem:leftcont}
The function $\exp:\R_+\to \mathcal{T}, \exp(r)= \ov{I^r}$ is left continuous.
\end{prop}
\begin{proof}
Fix $r\in \R_+$ and consider the set $A_r=\N^d\setminus r\cdot NP(I)$. Observe that each point $\bp\in A_r$ lies at a positive Euclidean distance from any point in  $r\cdot NP(I)$. Indeed in the notation of \eqref{eq:A}, writing $\ba_i$ for the $i$-th row of $A$ we have $\ba_i\cdot \bp < rc_i$ for at least one $1\leq i\leq d$ and thus the distance from $\bp$ to the hyperplane of equation $\ba_i\cdot\bx=rc_i$ is  $d_i=(rc_i-\ba_i\cdot \bp)/\sqrt{\ba_i\cdot\ba_i}>0$. In particular, since $\ba_i\cdot \bp\in\N$,  it follows that $d_i\geq \delta_i:=(rc_i-{\rm prec}(rc_i))/\sqrt{\ba_i\cdot\ba_i}>0$, where ${\rm prec}(u)$ is the largest integer strictly smaller than $u$. Taking $\Delta=\min_{1\leq i\leq d} \delta_i$ we conclude that any $\bp\in A_r$ lies at distance at least $\Delta>0$ from any point in $NP(I)$.

Since each $\delta_i$ is a left continuous function of $r$, it follows that there exists $\varepsilon_0>0$ such that for any $0<\varepsilon<\varepsilon_0$ each point $\bp\in A_r$ lies at a positive Euclidean distance from any point in  $(r-\varepsilon)\cdot NP(I)$ as well. Equivalently we have $A_r\cap (r-\varepsilon)\cdot NP(I)=\emptyset$ which yields  $A_r=A_{r-\varepsilon}$ and thus $r\cdot NP(I)\cap \N^d=(r-\varepsilon)\cdot NP(I) \cap \N^d$ and $\ov{I^r}=\ov{I^{r-\varepsilon}}$ for $0<\varepsilon<\varepsilon_0$. 
\end{proof}

We now show that the real exponentiation function of a monomial ideal is a step function.

\begin{cor}
\label{cor:step}
Let $j< j'$ be two consecutive jumping numbers for $I$. Then the function $\exp:\R_+\to \mathcal{T}, \exp(r)= \ov{I^r}$ is constant on $(j,j']$ and $\ov{I^j}\neq \ov{I^{j'}}$.
\end{cor}
\begin{proof}
Since $j< j'$ are consecutive jumping numbers, meaning there is no jumping number in $(j,j')$, the exponentiation function is continuous on $(j,j']$ by a combination of \Cref{lem:I>r=I^t_n} and \Cref{lem:leftcont}, and left continuity at $j'$. Since $\mathcal{T}$ carries the discrete topology, this continuity is equivalent to the function being constant on $(j,j']$. However, the exponentiation function is right discontinuous at $j$ by the definition of jumping number, thus $\ov{I^j}$ is distinct from the common value of the exponentiation function on $(j,j']$, that is,  $\ov{I^j}\neq \ov{I^{j'}}$.
\end{proof}

Our next aim is to show that the jumping numbers for monomial ideals are rational. Utilizing the notation in \eqref{eq:A}  and setting $\ba_i$ to be the $i$-th row of the matrix $A$ therein, the facets of the Newton polyhedron are supported on hyperplanes $H_i$ with equation $\ba_i\bx=c_i$. Each facet $F_i$ of $NP(I)$ is thus cut out by a system formed by one equation and several inequalities of the form
\begin{equation}
\label{eq:Fi}
F_i=\left\{\bx \mid \ba_i \cdot \bx=c_i,
  \ba_j \cdot \bx \geq c_j \text{ for } 1\leq j\leq d, j\neq i \right \}.
\end{equation}

\begin{prop}
\label{lem:jumpHalfSpaces}
Given a monomial ideal $I$ with facets $F_i, 1\leq i\leq s$ for $NP(I)$ described as in \eqref{eq:Fi} above, the following are equivalent:
\begin{enumerate}
    \item $r\in \R_+$ is a jumping number for $I$;  
    \item for some $1\leq i\leq s$ such that $c_i\neq 0$ there exists a lattice point $\bp\in r\cdot F_i\cap \N^d$;
    \item for some $1\leq i\leq s$ such that $c_i\neq 0$ there exists an integer solution to the system of equations and inequalities that describes $r\cdot F_i$, namely
        \begin{equation}
            \label{eq:jumpsystem}
            \begin{cases}
           \ba_i \cdot \bx=rc_i\\
            \ba_j \cdot \bx \geq rc_j   \text{ for } 1\leq j\leq d, j\neq i.
            \end{cases}
        \end{equation}  
\end{enumerate}
\end{prop}

\begin{proof}
$(2)\Leftrightarrow (3)$ is clear.

$(1)\Rightarrow (2)$
We show the contrapositive. Assume that $r\in \R_+$ is such that the union of the  facets of $r\cdot F_i$ of $r\cdot NP(I)$ corresponding to $c_i\neq 0$ contains no lattice point. Note that
\[
r\cdot NP(I)\setminus (r+\varepsilon)\cdot NP(I)\subseteq \{\bx \mid \ba_i \cdot \bx \in [rc_i, (r+\varepsilon)c_i) \text{ for some } i \text{ with } c_i\neq 0\}.\]
Moreover there is at least one $1\leq i\leq d$ so that $c_i\neq 0$ since $I\neq R$. 
Taking  $\varepsilon< \varepsilon_0:=\min_{c_i\neq 0}\{ (\rm{next}(rc_i) -rc_i)/c_i\}$, where $\rm{next}(rc_i)$ is the smallest integer strictly larger than $rc_i$, one can ensure that $[rc_i, (r+\varepsilon)c_i)\cap \N\subseteq \{rc_i\}$ whenever $c_i\neq 0$. 
This means that any possible lattice point $\bt$ in $r\cdot NP(I)\setminus (r+\varepsilon)\cdot NP(I)$  satisfies $\ba_i\cdot \bt =rc_i$ for some $c_i\neq 0$. Thus we see that $\bt$ lies on a facet $F_i$ which  contains no lattice points by assumption. Thus  there are no lattice points in 
$r\cdot NP(I)\setminus (r+\varepsilon)\cdot NP(I)$.

It follows that $\ov{I^r}=\ov{I^{r+\varepsilon}}$ for $0<\varepsilon<\varepsilon_0$ and thus $r$ is not a jumping number for $I$.

$(3)\Rightarrow (1)$
Let $\mathbf{p}\in\N^d$ be an integer solution to \eqref{eq:jumpsystem}. Since this implies $\bp\in r\cdot F_i\subseteq r\cdot NP(I)$,  we see that $\bx^{\mathbf{p}} \in \overline{I^r}$. Since $\bp$ attains equality in the first equation of \eqref{eq:jumpsystem} it follows that $\bp$ satisfies $\ba_{i}\cdot \bp< (r+\varepsilon)c_i$ for any $\varepsilon>0$. (This uses $c_i\neq 0$.) Thus we conclude $\bp\not \in (r+\varepsilon)\cdot NP(I)$ and $\bx^{\mathbf{p}} \notin \overline{I^{r+\epsilon}}$ for all $\epsilon > 0$ and therefore $\bx^\bp \notin \overline{I^{>r}}$. Consequently $r$ is a jumping number.
\end{proof}

From the above characterization we obtain that jumping numbers control the behavior of all real powers of a given monomial ideal and are all rational numbers.

\begin{thm}
\label{thm:jumpingrational}
Let $I$ be a monomial ideal.
\begin{enumerate}
\item All jumping numbers for $I$ are rational.
\item All distinct real powers of $I$ are given by rational exponents, i.e., for each $r\in \R_+$ there exists $r'\in \Q$ so that $\ov{I^r}=\ov{I^{r'}}$. Moreover $r'$ can be taken to be a jumping number for $I$.
\item If $r$ is a jumping number of $I$ then $nr$ is a jumping number for all $n \in \N$.
\item If $\mathbf{v}$ is a vertex of $NP(I)$, then for all $n\in \N$ the number $r_n=\frac{n}{\gcd(v_1, \cdots, v_d)}$ is a jumping number of $I$ .
\item The set of jumping numbers can be written as a finite union of scaled monoids
$
\mathcal{J}= \bigcup_{c_i\neq 0} \frac{1}{c_i} S_i.
$
Here each $S_i$ is a submonoid of the numerical semigroup generated by the entries of the $i$-th row  of the matrix $A$  in \eqref{eq:A} and $c_i$ are the components of the vector $\bc$ in \eqref{eq:A}.
\end{enumerate}
\end{thm}
\begin{proof}
(1) follows since  \Cref{lem:jumpHalfSpaces} (3) yields that there is an integer solution $\bp$ to an equation of the form $\ba_i \cdot \bp=rc_i$ where $\ba_i$ is a row of the matrix $A$ in \eqref{eq:A} and $c_i\neq 0$. Since the entries of $\ba_i, \bp$, and $c_i$ are natural numbers, this gives $r\in\Q$.

(2)  If $r\in \Q_+$ is a jumping number, set $r'=r$. If $r$ is not a jumping number, let 
\[r'=\inf\{  u \mid u>r \text{ and } u \text{ is a jumping number for } I\}. \]
Notice first that $r'$ is in fact the minimum of the set above, equivalently  $r'\in \Q$ is a jumping number for $I$. Indeed, if this is not the case, then there is a sequence of pairwise distinct jumping numbers $\{u_n\}_{n\in \N}$ converging to $r'$ from the right. Since we have assumed $r'$ is not a jumping number,   the exponential function with base $I$ is right continuous at $r'$, thus it must be the case that $\ov{I^{u_n}}=\ov{I^{r'}}$ for $n\gg 0$. This contradicts that the numbers $u_n$ are distinct jumping numbers, since distinct jumping numbers yield distinct real powers by \Cref{cor:step}. Another application of  \Cref{cor:step} together with the observation that $r$ is not a jumping number yields that the exponentiation function is constant on $[r, r']$, thus we conclude there is an equality $\ov{I^r}=\ov{I^{r'}}$.

(3) follows since the condition on integer solutions to the system \eqref{eq:jumpsystem} in \Cref{lem:jumpHalfSpaces} is preserved upon scaling the system by any natural number.

(4) Each vertex $\bv$ of $NP(I)$ furnishes an integer solution to the system of (in)equalities \eqref{eq:jumpsystem} corresponding to each facet $F_i$ such that $\bv\in F_i$. Scaling by $r_n$ we see that $r_n\cdot \bv\in \N^d$ is an integer solution to the analogous system
\[    \begin{cases}
    \ba_i \cdot \bx=r_nc_i,\\
\ba_j \cdot \bx\geq r_nc_j \text{ for } 1\leq j\leq d, i\neq j.
    \end{cases}
    \]
 \Cref{lem:jumpHalfSpaces} yields that $r_n$ is a jumping number for $I$.

For (5), for each $1\leq i\leq d$, let $a_{ij}\in\N$ be the entries in the $i$-th row of the matrix $A$ in  \eqref{eq:A} and $c_i$ the entries of $\bc$. For each $i$ with $c_i\neq 0$ set 
\[
S_i=\left \{rc_i \mid  \exists x_1,\ldots, x_d\in \N\cup\{0\} \text{ s.t.}\sum_{j=1}^d a_{ij}x_j= rc_i, \sum_{j=1}^d a_{lj}x_j\geq  rc_l \text{ for } l\neq i \right \}.
\]
 It is clear that $S_i\subset \N\cup\{0\}$. Moreover $S_i$ is a monoid as $0\in S_i$ and $rc_i, r'c_i\in S_i$ imply $(r+r')c_i\in S_i$ by summing the respective (in)equalities. The existence of a non-negative integer solution to the equation $\sum_{j=1}^d a_{ij}x_j= rc_i$ implies that $rc_i$ belongs to the numerical semigroup $M_i$ generated by the integers $a_{ij}$ for $1\leq j\leq d$, thus $S_i\subseteq M_i$. 
With this notation, \Cref{lem:jumpHalfSpaces} can be rephrased to say that the set of jumping numbers for $I$ is 
\[
\mathcal{J}= \bigcup_{c_i\neq 0} \frac{1}{c_i} S_i.
\]
\end{proof}

In regards to item (1) of \Cref{thm:jumpingrational} we observe that every non-negative rational number is a jumping number for some monomial ideal. Indeed if $r=\frac{p}{q}$ with $p,q\in \N, q\neq 0$ then $r$ is a jumping number of $I=(x_1^q)$.

Item (2) of \Cref{thm:jumpingrational} yields a new description for the image of the exponentiation function with base $I$
\[
\mathcal{T}=\{\ov{I^r}\mid r\in \Q\text{ is a jumping number for } I\}.
\]
Moreover, the elements of the set $\mathcal{T}$ listed above are pairwise distinct by \Cref{cor:step}.

We end with a worked out example which illustrates the jumping numbers and real powers of a particular monomial ideal using the criterion in \Cref{lem:jumpHalfSpaces}.

\begin{ex}
    The monomial ideal $I = (x^9, x^4y^3,x^2y^5,y^8)$ has Newton polyhedron depicted in \Cref{fig:coloredboundaries} with vertices at $(9,0),(4,3), (2,5), (0,8)$.
 
 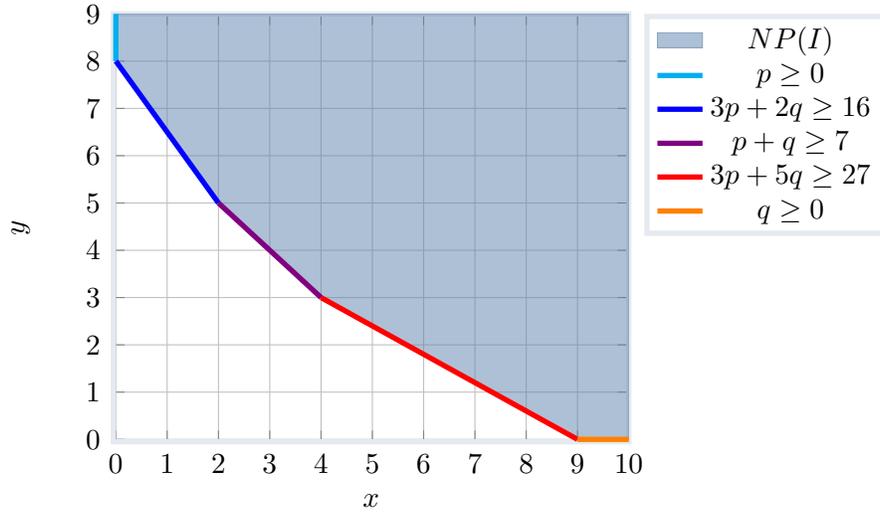
\begin{figure}[h!]
    \centering
    %  I = (x^9, x^4y^3,x^2y^5,y^8), r = 1
    \begin{tikzpicture}
        \tikzstyle{every node}=[font=\small]
        \pgfplotsset{every axis legend/.append style={legend pos=outer north east,draw=NordWhite}}
        
        %  % plots in legend
        \begin{axis}
        [posQuad, xtick={0,...,10}, ytick={0,...,9}, 
        xmin=-0.05, xmax=10, ymin=-0.05, ymax=9]
        
        % newton polyhedron region 
        \addplot[area style, fill=NordBrightBlue, opacity=0.5, line width=0pt, draw = NordBrightBlue] coordinates{(0,8) (2,5) (4,3) (9,0) (10,0) (10,9) (0,9)};
        \addlegendentry{$NP(I)$}
        
        % boundaries of the newton polyhedron
        \addplot[cyan] coordinates{(0,9) (0,8)};
        \addlegendentry{$p\geq 0$}

        \addplot[blue] coordinates{(0,8) (2,5)};
        \addlegendentry{$3p+2q\geq 16$}
        
        \addplot[violet] coordinates{(2,5) (4,3)};
        \addlegendentry{$p+q\geq 7$}
        
        \addplot[red] coordinates{(4,3) (9,0)};
        \addlegendentry{$3p+5q\geq 27$}
        
        \addplot[orange] coordinates{(9,0) (10,0)};
        \addlegendentry{$q\geq 0$}
        \end{axis}

    \end{tikzpicture}
    \caption{The Newton polyhedron of $(x^9, x^4y^3,x^2y^5,y^8)$}
    \label{fig:coloredboundaries}
\end{figure}
 
 We show that the jumping numbers of $I$ are the elements of the following set 
 \begin{equation}
 \label{eq:jump}
\mathcal{J}=\left \{ 0, \frac{i}{7}, \frac{j}{16}, \frac{k}{27}   \mid i\geq 2, j\in\{2,4,6\} \text{ or } j\geq 8, k\in\{3,6,9,11,12,14,15\} \text{ or } k\geq 17 \right\}.
\end{equation}
The  faces of the Newton polyhedron $\textcolor{cyan}{F_1}, \textcolor{violet}{F_2}, \textcolor{blue}{F_3}, \textcolor{red}{F_4}, \textcolor{orange}{F_5}$ are shown in \Cref{fig:coloredboundaries} together with the corresponding bounding inequalities for $NP(I)$. 
Putting these inequalities in the form of \eqref{eq:A} yields
\[
\begin{bmatrix}
1 & 0 \\
3 & 2 \\
1 & 1 \\
3 & 5 \\
0 & 1 
\end{bmatrix}
\cdot 
\begin{bmatrix}
p \\ q 
\end{bmatrix}
\geq 
\begin{bmatrix}
0 \\ 16 \\7 \\27 \\0
\end{bmatrix}.
\]
%Using the respective equations and the values $\min(F_1,1)=2, \max(F_1,1)=4, \min(F_1,2)=3, \max(F_1,2)=5$, $\min(F_2,1)=0, \max(F_2,1)=2, \min(F_2,2)=5, \max(F_2,2)=8, \min(F_3,1)=4$, $\max(F_3,1)=9, \min(F_3,2)=0, \max(F_3,2)=3$
%with respect to coordinates $p,q$ for $\R^2$ are 
% \[
% \begin{cases}
% p+q\geq 7\\
% 3p+2q\geq 16\\
% 3p+5q\geq 27\\
% p\geq 0, q\geq 0.
% \end{cases}
% \]
\Cref{thm:jumpingrational} (5) says that the jumping numbers depend on the following three monoids: 
 \begin{eqnarray*}
 S_2&=&\{ 16 r \mid \exists p,q \in \N\cup\{0\} \text{ s.t. } p\geq 0, 3p+2q=16r, p+q \geq 7r, 3p+5q\geq 27r, q\geq 0\},\\
 S_3&=&\{ 7 r \mid \exists p,q \in \N\cup\{0\} \text{ s.t. } p\geq 0, 3p+2q\geq16r, p+q = 7r, 3p+5q\geq 27r, q\geq 0\},\\
S_4&=&\{ 27 r \mid \exists p,q \in \N\cup\{0\} \text{ s.t. } p\geq 0, 3p+2q\geq16r, p+q \geq 7r, 3p+5q= 27r, q\geq 0\}.
 \end{eqnarray*}
 Denote $\N_0=\N\cup\{0\}$.
 It turns out that $S_2=2\N_0+9\N_0$, $S_3= 2\N_0+ 3\N_0$, and $S_4=3\N_0+11\N_0+19\N_0$.   The set of jumping numbers is
 \[
 \mathcal{J}= \frac{1}{16} S_2 \cup \frac{1}{7} S_3 \cup \frac{1}{27} S_4,
 \]
Writing the the elements of each semigroup $S_1, S_2, S_3$ explicitly yields the set displayed in equation \eqref{eq:jump} above.

Below we list the rational powers of $I$ for exponents $r\in (0,1]$. The generators have been color coded based on the bounded edge of the Newton polyhedron that is giving rise to change in generator(s) cf.~\Cref{lem:jumpHalfSpaces} (2). Refer to the legend in \Cref{fig:coloredboundaries} for the color corresponding to each edge.  

    \[
    \overline{I^r} = 
    \begin{cases}
        (y,x)
        & r \in ( 0             , \frac{1}{9} ]    \\
        
        ( y, \textcolor{orange}{x^2})
        & r \in ( \frac{1}{9}   , \frac{1}{8} ] \\
        
        ( \textcolor{cyan}{y^2},  \textcolor{cyan}{xy}, x^2)                  
        & r \in ( \frac{1}{8}   , \frac{2}{9} ]   \\
        
        (y^2,  xy, \textcolor{orange}{x^3})
        & r \in ( \frac{2}{9}   , \frac{2}{8} ] \\
        
        ( \textcolor{cyan}{y^3},  xy, x^3)
        & r \in ( \frac{2}{8}   , \frac{2}{7} ] \\
        
        (y^3, \textcolor{violet}{xy^2}, \textcolor{violet}{x^2y},x^3)          
        & r \in ( \frac{2}{7}   , \frac{3}{9} ]   \\
        
        (y^3, xy^2, x^2y,\textcolor{orange}{x^4})
        & r \in ( \frac{3}{9}   , \frac{3}{8} ] \\
        
        (\textcolor{cyan}{y^4}, xy^2, x^2y,x^4)
         & r \in ( \frac{3}{8}   , \frac{11}{27} ] \\
        
        (y^4, xy^2, \textcolor{red}{x^3y},x^4)
        & r \in ( \frac{11}{27}   , \frac{3}{7} ] \\
        
        ( y^4, \textcolor{violet}{xy^3}, \textcolor{violet}{x^2y^2}, x^3y,x^4)
        & r \in ( \frac{3}{7}   , \frac{4}{9} ]  \\
        
        (y^4, xy^3, x^2y^2, x^3y,\textcolor{orange}{x^5})
        & r \in ( \frac{4}{9}   , \frac{4}{8} ] \\
        
        (\textcolor{cyan}{y^5}, xy^3, x^2y^2, x^3y,x^5)
        & r \in ( \frac{4}{8}   , \frac{14}{27} ] \\
        
        ( y^5, xy^3, x^2y^2, \textcolor{red}{x^4y},x^5)
        & r \in ( \frac{14}{27}   , \frac{5}{9} ]\\
        
        (y^5, xy^3, x^2y^2, x^4y,\textcolor{orange}{x^6})
        & r \in ( \frac{5}{9}   , \frac{9}{16} ]\\
        
        (y^5, \textcolor{blue}{xy^4}, x^2y^2, x^4y,x^6)
        & r \in ( \frac{9}{16}   , \frac{4}{7} ]\\
        
        (y^5, xy^4, \textcolor{violet}{x^2y^3}, \textcolor{violet}{x^3y^2}, x^4y,x^6)
        & r \in ( \frac{4}{7}   , \frac{5}{8} ] \\
        
        (\textcolor{cyan}{y^6}, xy^4, x^2y^3, x^3y^2, x^4y,x^6)
        & r \in ( \frac{5}{8}   , \frac{17}{27} ] \\
        
        (y^6, xy^4, x^2y^3, x^3y^2, \textcolor{red}{x^5y},x^6)
        & r \in ( \frac{17}{27}   , \frac{6}{9} ] \\
        
        ( y^6, xy^4, x^2y^3, x^3y^2, x^5y,\textcolor{orange}{x^7})
        & r \in ( \frac{6}{9}   , \frac{11}{16} ] \\
        
        ( y^6, \textcolor{blue}{xy^5}, x^2y^3, x^3y^2, x^5y,x^7)
        & r \in ( \frac{11}{16}   , \frac{19}{27} ] \\
        
        ( y^6, xy^5, x^2y^3, \textcolor{red}{x^4y^2} , x^5y,x^7)
        & r \in ( \frac{19}{27}   , \frac{5}{7} ]\\
        
        ( y^6, xy^5, \textcolor{violet}{x^2y^4}, \textcolor{violet}{x^3y^3}, x^4y^2, x^5y,x^7)
        & r \in ( \frac{5}{7}   , \frac{20}{27}]\\
        
        (y^6, xy^5,x^2y^4, x^3y^3, x^4y^2,\textcolor{red}{x^6y},x^7) & r\in(\frac{20}{27},\frac{3}{4}]\\
        
        (\textcolor{cyan}{y^7}, xy^5,x^2y^4, x^3y^3, x^4y^2,x^6y,x^7) & r\in(\frac{3}{4},\frac{7}{9}]\\
        
        (y^7, xy^5,x^2y^4, x^3y^3, x^4y^2,x^6y,\textcolor{orange}{x^8}) & r\in(\frac{7}{9},\frac{13}{16}]\\
        
        (y^7, \textcolor{blue}{xy^6},x^2y^4, x^3y^3, x^4y^2,x^6y,x^8) & r\in(\frac{13}{16},\frac{22}{27}]\\
        
        (y^7, xy^6,x^2y^4, x^3y^3, \textcolor{red}{x^5y^2},x^6y,x^8) & r\in(\frac{22}{27},\frac{23}{27}]\\
        
        (y^7, xy^6,x^2y^4, x^3y^3, x^5y^2,\textcolor{red}{x^7y},x^8) & r\in(\frac{23}{27},\frac{6}{7}]\\
        
        (y^7, xy^6,\textcolor{violet}{x^2y^5}, \textcolor{violet}{x^3y^4}, \textcolor{violet}{x^4y^3}, x^5y^2,x^7y,x^8) & r\in(\frac{6}{7},\frac{7}{8}]\\
        
        (\textcolor{cyan}{y^8}, xy^6,x^2y^5, x^3y^4, x^4y^3, x^5y^2,x^7y,x^8) & r\in(\frac{7}{8},\frac{8}{9}]\\
        
        (y^8, xy^6,x^2y^5, x^3y^4, x^4y^3, x^5y^2,x^7y,\textcolor{orange}{x^9}) & r\in(\frac{8}{9},\frac{25}{27}]\\
        
        (y^8, xy^6,x^2y^5, x^3y^4, x^4y^3, \textcolor{red}{x^6y^2},x^7y,x^9) & r\in(\frac{25}{27},\frac{15}{16}]\\
        
        (y^8, \textcolor{blue}{xy^7},x^2y^5, x^3y^4, x^4y^3, x^6y^2,x^7y,x^9) & r\in(\frac{15}{16},\frac{26}{27}]\\
        
        (y^8, xy^7,x^2y^5, x^3y^4, x^4y^3, x^6y^2,\textcolor{red}{x^8y},x^9) & r\in(\frac{26}{27},1]\\
    \end{cases}
    \]

\end{ex}

%\vspace{1em}

{\bf Acknowledgement.} We warmly thank the anonymous referees for providing careful and detailed comments which greatly improved the paper and in particular for their contributions to  \Cref{lem:intclosure} and \Cref{lem:jumpHalfSpaces}. 

This work was completed in the framework of the 2020 Polymath Jr. program \href{https://geometrynyc.wixsite.com/polymathreu}{https://geometrynyc.wixsite.com/polymathreu}. The second author was supported by the NSF RTG grant in algebra and combinatorics at the University of Minnesota  DMS--1745638. The fourth author was supported by NSF DMS--2101225.
\bibliographystyle{amsalpha}
%\bibliography{references}

\providecommand{\bysame}{\leavevmode\hbox to3em{\hrulefill}\thinspace}
\providecommand{\MR}{\relax\ifhmode\unskip\space\fi MR }
% \MRhref is called by the amsart/book/proc definition of \MR.
\providecommand{\MRhref}[2]{%
  \href{http://www.ams.org/mathscinet-getitem?mr=#1}{#2}
}
\providecommand{\href}[2]{#2}

\end{document}